\documentclass{amsart}

\usepackage{amssymb}

\newtheorem{thm}{Theorem}[section]

\newtheorem{lemma}[thm]{Lemma}

\theoremstyle{remark}
\newtheorem{rem}[thm]{Remark}

\theoremstyle{definition}

\numberwithin{equation}{section}
\numberwithin{thm}{section}

\newcommand{\N}{\mathbb{N}}
\newcommand{\R}{\mathbb{R}}
\newcommand{\D}{\mathbb{D}}
\newcommand{\B}{\mathbb{B}}
\newcommand{\C}{\mathbb{C}}

\newcommand{\diver}{\operatorname{div}}
\newcommand{\sgn}{{\rm{sgn}}}

\newcommand{\abs}[1] {\lvert #1 \rvert}

\def\Xint#1{\mathchoice
    {\XXint\displaystyle\textstyle{#1}}%
    {\XXint\textstyle\scriptstyle{#1}}%
    {\XXint\scriptstyle\scriptscriptstyle{#1}}%
    {\XXint\scriptscriptstyle\scriptscriptstyle{#1}}%
    \!\int}
\def\XXint#1#2#3{{\setbox0=\hbox{$#1{#2#3}{\int}$}
      \vcenter{\hbox{$#2#3$}}\kern-.5\wd0}}
\newcommand{\vint}{\Xint-}

\begin{document}

%%%%%%%%%%%%%%%%%%%%%%%%%
% Subject classification 
%%%%%%%%%%%%%%%%%%%%%%%%%

% Provide an AMS subject classification with one or two primary classification 
% numbers and, optionally, one or more secondary classification numbers. 
% Use the following format:  "Primary 42B25. Secondary 42B60, 20E26"

\subjclass{Primary 35J70. Secondary 35J25, 35J92}

%%%%%%%%%
% Title
%%%%%%%%%

% Title, in lower case, with no explicit linebreaks (\\).  If the title
% is too long to be used as a running head, add a short version of the
% title in brackets, as in \title[shorttitle]{fulltitle}.

\title[on a linearized $p$-Laplace equation]{on a linearized $p$-Laplace equation with rapidly oscillating
coefficients}

%%%%%%%%%%%%%%%%%%%%%%%%%%%%%%
% Author names and addresses 
%%%%%%%%%%%%%%%%%%%%%%%%%%%%%%

% Provide one separate \author{...} \address{...} \email{....} entry for each
% author, i.e., do not combine multiple authors.  Separate address lines by double
% slashes.  Do not attach footnotes to author  names. (For acknowledgements use
% the "\thanks" construct below.)
%

\author{Harri Varpanen}
\address{
 Department of Mathematics and Systems Analysis,
 Aalto University,
 PO Box 11100, 00076 Aalto, Finland.
}

\email{harri.varpanen@aalto.fi}

%%%%%%%%%%%%%%%%%%%%
% Acknowledgements
%%%%%%%%%%%%%%%%%%%

% Use \thanks for acknowledgements as footnotes to the title page.  
% (Note that footnotes inside \author or \title macros are not
% allowed.)
%
% In case of multiple author papers, phrase the acknowledgement to 
% say "The first author was supported by ...  The second author was
% supported by ..."

% \thanks{}

%%%%%%%%%%%%%
% Abstract 
%%%%%%%%%%%%%
%
% Abstracts should not contain macros (so that they can be processed independently
% of the paper.) Avoid displayed math and references in the abstract.

\begin{abstract}
Related to a conjecture of Tom Wolff, we solve a singular Neumann problem
for a linearized $p$-Laplace equation in the unit disk.
\end{abstract}

\maketitle

%%%%%%%%%%%%%%%%%%%%%%%%%%%%%%%%%%%%%%%%%%%%%%%%%%%%%%%%%%%%%%%%%%%%%%%%%
% end Topmatter
%%%%%%%%%%%%%%%%%%%%%%%%%%%%%%%%%%%%%%%%%%%%%%%%%%%%%%%%%%%%%%%%%%%%%%%%%

%%%%%%%%%%%%%%%%%%%%%%%%%%%%%%%%%%%%%%%%%%%%%%%%%%%%%%%%%%%%%%%%%%%%%%%%%
% body of paper
%%%%%%%%%%%%%%%%%%%%%%%%%%%%%%%%%%%%%%%%%%%%%%%%%%%%%%%%%%%%%%%%%%%%%%%%%

\section{Introduction}
Tom Wolff \cite{Wolff:1984} constructed in 1984 a celebrated example of a bounded
$p$-harmonic function $u$ in the upper half-plane
$\R^2_+ = \{(x,y)\in \R^2\colon y > 0\}$ such that the set
\begin{equation*}
\{x \in \R \colon \lim_{y\to 0} u(x,y) \text{ exists }\}
\end{equation*}
has $1$-dimensional Lebesgue measure zero.
Fatou's classical radial limit Theorem \cite{Fatou:1904, Stein:1970} %(for a modern treatment see \cite{Stein:1970})
states that any bounded harmonic function in a smooth Euclidean domain has nontangential limits almost everywhere on the boundary of the domain, so
Wolff's construction demonstrates the failure of Fatou's Theorem in the nonlinear case $p \ne 2$.

The most important ingredient in Wolff's argument (see \cite[Lemma 1]{Wolff:1984}) is the construction 
 of a bounded $p$-harmonic function $\Phi = \Phi(x,y)$ in $\R^2_+$
such that $\Phi$ has 
period $\lambda = \lambda_p$ in the $x$ variable, $\Phi(x,y) \to 0$ as
$y \to \infty$ (uniformly in $x$), and \label{Phi}
\begin{equation*}
\int_0^\lambda \Phi(x,0)\,dx \ne 0.
\end{equation*}
This is essentially a failure of the mean value principle; no such construction can
be done for harmonic functions. See Section \ref{fatou_failure} below for a description of how the failure of Fatou's Theorem follows.

Wolff states \cite[page 372]{Wolff:1984} that his argument
{\it must generalize to other domains}, and that {\it the argument is easiest
in a half-space since the $p$-Laplace operator behaves well under Euclidean
operations.} 
It is particularly interesting whether a construction is
possible in bounded domains such as the unit disk, because in general there are serious problems when trying to map an unbounded planar
domain on a bounded one in the
$p$-harmonic setting. Conformal invariance is lost, and there cannot exist any reasonable counterpart to the Kelvin transform when $p \ne 2$;
see \cite{Lindqvist:1989}.
While many open problems for $p$-harmonic functions
are resolved in two dimensions, boundary behavior is still far from
well understood.

This paper is motivated by the problem of whether it
is possible to construct a bounded $p$-harmonic function $u = u(r,\theta)$ in the unit disk such that the set
\[
\left\{ \theta \in [0,2\pi) \colon \lim_{r \to 1} u(r,\theta) \text{ exists}\right\}
\]
has measure zero. A closely related problem is to construct a sequence $\left(u_N\right)_{N=1}^\infty$ of bounded $p$-harmonic functions in the unit disk such
that the $N^{\text th}$ function has angular period $2\pi/N$ and for each function $u = u_N = u(r,\theta)$,
\[
u(0) \ne \vint_0^{2\pi} u(1,\theta)\ d\theta.
\]
The author claimed \cite{Varpanen:2008} to have constructed such a sequence, but the construction
contained a gap that remains open and is explained below in Section \ref{error_statement}.
% Also other authors \cite{LW:oral} have reported about a successful construction, but their results remain unpublished.
In this paper we construct the corresponding sequence for a linearized equation.

\subsection{Organization of the paper and statement of results.}
Section \ref{section2} contains preliminaries about $p$-harmonic functions, and Section \ref{section3} introduces the appropriate moving frame
intrinsic to the unit disk. In Section \ref{section4} we start with a well-known sequence $\left(f_N\right)_{N=1}^\infty$ of $p$-harmonic functions
in $\R^2$ with the polar form
\begin{equation}
\label{quasiradial}
f_N(r,\theta) = r^ka_N(\theta),
\end{equation}
where $k=k(p,N)>0$ and the function $a_N$ is $2\pi/N$-periodic. Unfortunately for our purposes, each of these functions $f = f_N$ satisfy
\[
f(0) = \vint_0^{2\pi}f(r,\theta) \ d\theta = 0 \quad \text{for each } r > 0,
\]
so a suitable perturbation is called for. We let $v = v_N$ be a solution to the
linearized $p$-Laplace equation
\begin{equation}
\label{linearized_plap}
\left.\frac{d}{d\varepsilon}\right |_{\varepsilon = 0} \Delta_p (f+\varepsilon v) = 0,
\end{equation}
where $f=f_N$ is as in \eqref{quasiradial}. After introducing the appropriate weighted function spaces in Section \ref{section5} and proving
a priori regularity in Section \ref{section6}, our main result, proved in Section \ref{section7}, is the following:
\begin{thm}
\label{thm1}
For a given $N \in \N$, there exists a solution $v = v_N$ to \eqref{linearized_plap} such that
\begin{equation}
\label{neumann7}
\int_0^{2\pi} \frac{\partial v}{\partial n}(1,\theta)d\theta  \ne 0,
\end{equation}
where $n$ denotes the outer unit normal.
\end{thm}
As described below in Section \ref{Neumann_MVP}, this would yield the desired failure of the mean value principle for $p$-harmonic functions, provided that
the function $v = v_N$ in \eqref{neumann7} satisfied $\nabla v \in L^p(\D)$. 
We conjecture this to be true, but the best regularity we obtain is the following (Section \ref{section6}):
\begin{thm}
\label{thm2}
Any solution $v = v(x)$ to \eqref{linearized_plap} satisfies $v \in L^\infty(\D)$ and $\abs{x}\nabla v (x) \in L^\infty(\D)$.
\end{thm}
We prove Theorem \ref{thm2} by noticing that \eqref{linearized_plap} is the Euler-Lagrange equation of a quadratic energy functional,
and thereby we are able to utilize Wolff's results in the upper half-plane via a conformal map.

\begin{rem}
The equation \eqref{linearized_plap} works out to
\[
\diver(A\nabla v) = 0,
\]
where
\begin{equation*}
%\label{A}
\begin{split}
& A(r,\theta) \\
& = r^{(p-2)(k-1)}(a_\theta^2 + k^2a^2)^\frac{p-4}{2}
\begin{pmatrix}
a_\theta^2 + (p-1)k^2a^2 & (p-2)kaa_\theta \cr
(p-2)kaa_\theta & k^2a^2 + (p-1)a_\theta^2 \cr
\end{pmatrix},
\end{split}
\end{equation*}
the function $a$ is from \eqref{quasiradial}, and $a_\theta$ denotes its derivative. The equation \eqref{linearized_plap} is
degenerate / singular at the origin, and in such cases regularity of solutions is not well understood, see e.g. \cite{Zhong}.
\end{rem}

The outline of our treatment if similar to that of \cite[Sec. 3]{Wolff:1984}.
We give many additional details on various calculations that are only sketched in \cite{Wolff:1984}.
We use a moving frame, because using plain polar coordinates would render the Neumann problem in Section \ref{section7} difficult to solve.
% Many of the calculations are improved versions of those in \cite{Varpanen:2008}; Section \ref{section6} is new.

\subsection{How the failure of the mean value principle leads to the failure of Fatou's Theorem}
\label{fatou_failure}
Having constructed the function $\Phi = \Phi(x,y)$ described above on page \pageref{Phi}, Wolff's
argument in the half-plane continues roughly as follows:
%where $\widehat{g}$ denotes the unique $p$-harmonic function
%with boundary values $g$ (see Section \ref{section2}):

\begin{itemize}
\item If $(T_j)_{j=1}^\infty$ is a suitably fast-growing sequence of positive real numbers and 
if $(L_j)_{j=1}^\infty$ is a suitable sequence of uniformly bounded Lipschitz functions,
then the sequence 
\begin{equation}
\label{sigma_series}
\sigma_k(x) = \sum_{j=1}^k \frac{1}{j}L_j(x)\Phi(T_jx,0)
\end{equation}
is uniformly bounded and diverges for almost every $x$ as $k\to\infty$ (\cite[Lemma 2.12]{Wolff:1984}).
\item For each $k \ge 1$, denote by $\widehat{\sigma}_k$ the unique $p$-harmonic function in $\R^2_+$
having boundary values $\sigma_k$. Lemma 1.6 in \cite{Wolff:1984} enables one to fix the sequences $(T_j)_{j=1}^\infty$ and $(L_j)_{j=1}^\infty$,
and to obtain a decreasing sequence of positive numbers $\beta_k \to 0$ along with the following estimates:
\[
\left| \widehat{\sigma}_{k+1}(x,y)-\widehat{\sigma}_k(x,y) \right| < \frac{1}{2^k} \qquad \text{when } y>\beta_k
\]
\[
\left| \widehat{\sigma}_{k+1}(x,y)-\sigma_k(x) \right| < \frac{1}{k} \qquad \text{when } y\le \beta_k.
\]
\item It follows that the sequence $\widehat{\sigma}_k$ converges to a $p$-harmonic limit function $G$ as $k \to \infty$,
and that for a.e. $x$ the limit $\lim_{y\to 0} G(x,y)$ does not exist; see \cite[p. 385]{Wolff:1984}.
\end{itemize}

{\em Remark.} Denoting $\phi_j(x) = \Phi(T_jx,0)$ one has $\widehat{\phi}_j(x,y) = \Phi(T_jx, T_jy)$, so in the upper half-plane it is enough to construct
a single function failing the mean value principle. An analogous scaling with respect to the angular variable does not hold in the disk.

\subsection{How the Neumann problem leads to the failure of the mean value principle}
\label{Neumann_MVP}
In Section \ref{section7} we find a solution $v \in Y_1$ to the linearized $p$-Laplace equation such that
\begin{equation}
\label{bdaryvals}
\left|\int_0^{2\pi}\frac{\partial v}{\partial n}(1,\theta)\,d\theta\right|
= M>0.
\end{equation}
By the Fundamental Theorem of Calculus, there exists a radius $r_0$ close to one such that
\begin{equation*}
\left|\int_0^{2\pi}v(r_0,\theta)-v(1,\theta)\,d\theta\right|
> \frac{M}{2}(1-r_0),
\end{equation*}
especially
\begin{equation}
\label{differ_boundary}
\int_0^{2\pi}v(r_0,\theta)\,d\theta \ne \int_0^{2\pi}v(1,\theta)\,d\theta.
\end{equation}
Assuming $v$ is continuous at the origin, both of the functions $v_1 = v(r,\theta)$ and $v_2= v(r_0r,\theta)$ have the same
value at the origin. We conclude from (\ref{differ_boundary}) that
\[
v_i(0) \ne \int_0^{2\pi} v_i(1,\theta)\ d\theta
\]
for some $i\in\{1,2\}$.

{\em Remark.} Assuming $\nabla v \in L^p(\D)$, a similar argument
 holds for the $p$-harmonic function $\widehat{f+\varepsilon v}$ for
a small $\varepsilon$, even without the assumption that $v$ is continuous at the origin.
See \cite[Lemmas 3.16--3.19]{Wolff:1984}, where $\nabla v \in L^\infty(\R^2_+)$ can be replaced by $\nabla v \in L^p(\R^2_+)$ but any weaker regularity does not
seem to suffice. The calculations transfer verbatim to the disk case.

\subsection{Statement of error}
\label{error_statement}
It was claimed by the author 
\cite[Lemma 7.7]{Varpanen:2008} that $v \in C^0(\D) \cap W^{1,\infty}(\D)$ holds for any solution to the linearized $p$-Laplace equation
\eqref{linearized_plap}.
The argument claimed uniform ellipticity in dyadic annuli near the origin, but in fact the gradients of the test functions in Caccioppoli-type inequalities do not stay bounded.

{\em Conjecture.} We conjecture, based on numerical experiments, that for each $N \in \N$ there exists a solution $v = v_N$ to the linearized $p$-Laplace
equation \eqref{linearized_plap} such that \eqref{bdaryvals} holds and such that $v \in C^0(\D) \cap W^{1,\infty}(\D)$.

\section{Preliminaries}
\label{section2}
The $p$-Laplace equation $\Delta_p u = 0$, i.e.
\[
\operatorname{div}(|\nabla u|^{p-2}\nabla u) = 0,
\]
is the Euler-Lagrange equation for the variational integral
\begin{equation}
\label{p-dir}
I(u) = \int_\Omega |\nabla u|^p\ dx,
\end{equation}
where $1 < p < \infty$ and $\Omega$ is an Euclidean domain. A real-valued function $u \in W^{1,p}_\text{loc}(\Omega)$ is a solution of the $p$-Laplace equation
if and only if
\[
\int_\Omega \langle |\nabla u|^{p-2}\nabla u, \ \nabla \varphi \rangle\ dx = 0
\]
for all $\varphi \in C^\infty_0(\Omega)$.
A standard density argument (e.g. \cite[Lemma 3.11]{HKM:book})
yields that the class of test functions can be extended to $W^{1,p}_0(\Omega)$.
In general $p$-harmonic functions are not twice continuously differentiable,
but their first partial derivatives are locally H\"older continuous. A complete regularity characterization in the plane
was given in \cite{Iwaniec_Manfredi:1989}; in higher dimensions the optimal regularity is unknown.
Outside critical points $p$-harmonic are real analytic \cite{Lewis:1977}.

For a given $g \in W^{1,p}(\Omega)$ there exists a unique $p$-harmonic
function $u$ in $\Omega$ such that $u-g \in W^{1,p}_0(\Omega)$.
Equivalently, $u$ is the unique minimizer of the
$p$-Dirichlet integral $I(v)$ in \eqref{p-dir}
among functions $v \in W^{1,p}(\Omega)$ with $v-g\in W^{1,p}_0(\Omega)$;
see e.g. \cite{Lindqvist:notes}.

There are two planar analogues between the cases $p=2$ and $p \ne 2$
that we find particularly interesting: 

1) While harmonic functions are characterized by the asymptotic mean value property
\[
u(x) = \vint_{B_r(x)} u(y)\ dy + o(r^2) \quad \text{as } r \to 0
\]
for each $x \in \Omega$ and each ball $B_r(x) \subset \Omega$, $p$-harmonic functions are analogously characterized by
\[
u(x) = \alpha \ \vint_{B_r(x)} u(y)\ dy + (1-\alpha)\ \frac{1}{2}\left(\max_{\overline{B}_r(x)}u + \min_{\overline{B}_r(x)} u \right) + o(r^2) \quad \text{as } r \to 0,
\]
where $\alpha = 4/(p+2)$; see \cite{Arroyo_Llorente:2015, LM:2014}.

2) If $u$ is harmonic in a simply connected planar domain, there exists a conjugate harmonic function $v$, unique up to a constant,
such that
\[
u_x = v_y, \qquad u_y = -v_x,
\]
$\langle \nabla u, \nabla v \rangle = 0$, and such that the map $F = u+iv$ is conformal.
Analogously, if $u$ is $p$-harmonic in a simply connected domain $\Omega \subset \R^2$, there exists a conjugate $q$-harmonic\footnote{$\ \ 1/p+1/q=1$} function $v$, unique up
to a constant, such that
\[
u_x = |\nabla v|^{q-2}v_y, \qquad u_y = -|\nabla v|^{q-2}v_x,
\]
$\langle \nabla u, \nabla v\rangle = 0$, and such that the map $F = u+iv$ is locally quasiregular outside the isolated set
$
\{x \in \Omega \colon \nabla u = \nabla v = 0\};
$
see \cite{Lindqvist:curvature}.

For example, Wolff has $p>2$ in \cite{Wolff:1984}, but Lewis \cite{Lewis:1986} reduced the case $1<p<2$ to Wolff's
result by using the conjugacy property 2) above. It may be worthwile to carry out our program in a complex setting with
both of the conjugate pairs simultaneously present. Moreover, since the purpose of our work is to construct a $p$-harmonic
function that fails the mean value principle in a specific way, the characterization 1) above could provide useful insights.

Throughout in what follows, we will assume $p>2$; this property is used in Lemma \ref{calculation}.
Our domain of interest will be the unit disk $\D = \{x \in \R^2 \colon |x| < 1\}$, and we denote
$\D^* = \{x \in \R^2 \colon 0 < |x| < 1\}$. We will use the partial derivative notation (e.g. $u_r$ and
$a_\theta$) also in the case of a single variable function.

\section{A moving frame}
\label{section3}

Let $(r,\theta)$ denote the polar coordinates in the plane, and define a
moving frame intrinsic to $\D^*$ by
\begin{equation*}
e_r = \frac{\partial}{\partial r}, \qquad e_\theta = \frac{1}{r}\frac{\partial}{\partial \theta}.
\end{equation*}
Let $f$ be a real-valued function defined outside the origin in the plane.
The intrinsic gradient of $f$ is formally defined as a vector 
\[
\nabla^\circ f(r,\theta) = \bigl(e_rf(r,\theta), e_\theta f(r,\theta)\bigr),
\]
i.e. $\nabla^\circ f = R(\theta)\nabla f$, where $R(\theta)$ is the
rotation by $\theta$ and $\nabla f$ is the gradient in cartesian coordinates.

The dual basis to $\{e_r, e_\theta\}$ is $\{dr, rd\theta\}$, and the volume
element is $dA = rdrd\theta$. The adjoints $e_r^*$ and $e_\theta^*$ are defined
as
\begin{equation*}
\int_{\D^*} e_r(u)v\,dA = \int_{\D^*} u e_r^*(v)\,dA \quad \text{ for all }\,  u,v\in C^\infty_0(\D^*),
\end{equation*}
and similarly for $e_\theta$. The intrinsic divergence of a vector
field $F = (f, g)$ is formally defined as
\begin{equation*}
\operatorname{div}^\circ F = -\bigl(e^*_r(f) + e^*_\theta(g)\bigr).
\end{equation*}
\begin{lemma}
\label{correct_divergence}
Let $F = (f,g)$ be a differentiable vector field in $\D^*$. Then
\begin{equation*}
\operatorname{div}^\circ F =  \frac{1}{r}e_r(rf) + e_\theta(g).
\end{equation*}
\end{lemma}
\begin{proof}
The claim is
\begin{equation*}
e_r^*(f) = -\frac{1}{r}e_r(rf) \quad \text{ and} \quad e_\theta^*(g) = -e_\theta(g).
\end{equation*}

Let $\varphi \in C^\infty_0(\D^*)$. Since $(\varphi r f)_r = \varphi_r r f + \varphi(rf)_r$, we have
\begin{equation*}
\int f \varphi_r\,dA
= - \int\varphi\frac{1}{r}(rf)_r\,dA,
\end{equation*}
i.e. $e_r^*(f) = -\frac{1}{r} e_r(r f)$ as wanted. The $e_\theta$
case is similar: using $(g \varphi)_\theta = g\varphi_\theta + \varphi g_\theta$ leads to
\begin{equation*}
\int \frac{1}{r^2} g \varphi_\theta\,dA
= - \int\frac{1}{r^2}\varphi g_\theta\,dA,
\end{equation*}
i.e. $e_\theta^*(g) = -e_\theta(g)$.
\end{proof}
One now easily verifies that
$\Delta^\circ u := \operatorname{div}^\circ (\nabla^\circ u)
= \Delta u$ by writing the Laplacian in polar coordinates, i.e.
$$
\Delta^\circ u = u_{rr} + \frac{1}{r}u_r + \frac{1}{r^2}u_{\theta\theta},
$$
and the same holds for the $p$-Laplacian:
\begin{lemma}
The $p$-Laplace equation $\operatorname{div}(|\nabla u|^{p-2}\nabla u) = 0$
in $\D^*$, $1<p<\infty$, can be written as
\begin{equation}
\label{intrinsic_plap}
\operatorname{div}^\circ (|\nabla^\circ u|^{p-2}\nabla^\circ u) = 0
\ \text{ in } \ \D^*.
\end{equation}
\end{lemma}
\begin{proof}
Let $u \in W^{1,p}(\D^*)$ be $p$-harmonic in $\D^*$ and let
$\varphi \in C^\infty_0(\D^*)$.
Since $|\nabla u| = |\nabla^\circ u|$ and since
$\langle \nabla u , \nabla \varphi \rangle = \langle \nabla^\circ u , \nabla^\circ \varphi \rangle$,
the $p$-Laplace equation % (\ref{p_lap})
can be rewritten as
\begin{equation*}
\int_{\D^*} \abs{\nabla^\circ u}^{p-2}\langle \nabla^\circ u , \nabla^\circ \varphi\rangle\, dA = 0,
\end{equation*}
which is the weak form of (\ref{intrinsic_plap}).
\end{proof}

\section{A linearized $p$-Laplace equation}
\label{section4}

The following  well-known Lemma is adapted from Tkachev \cite{Tkachev:2006} and Aronsson \cite{Aronsson:1986, Aronsson:1989}.
% see Propn. 2.5, Cor. 2.3, and Thm. 3.2, respectively,
% of \cite{Tkachev:2006}.

\begin{lemma}
\label{lem:tkachev}
Given $1<p<\infty$ and $N\in\N$, there exists a $2\pi/N$-periodic function
$a = a_{p,N}\colon \R\to\R$ and a real number
$k = k(p,N) > 0$ such that the function
\begin{equation*}
f = f_{p,N}(r,\theta) = r^ka(\theta)
\end{equation*}
is $p$-harmonic in $\R^2$. Moreover, the following holds:
\begin{itemize}
\item[(i)] The number $k = k(p,N)$ is determined from the quadratic equation 
\[
(2N-1)(b + 1)k^2 -2(N^2b + 2N - 1)k + N^2(1+b) = 0,
\]
where $b=p/(p-2)$.
\item[(ii)] The function $a$ is characterized by the quasilinear ordinary
differential equation
\begin{equation}
\label{separation}
a_{\theta\theta} + V(a,a_\theta)a = 0,
\end{equation}
where
\begin{equation*}
V(a,a_\theta) = \frac{\bigl((2p-3)k^2-(p-2)k\bigr) a_\theta^2 + \bigl((p-1)k^2 -
(p-2)k\bigr) k^2a^2}{(p-1)a_\theta^2 + k^2a^2}.
\end{equation*}
\item[(iii)]
The equation (\ref{separation}) has a unique solution
$a \in C^\infty(\R)$ with given initial data $a(0),a_\theta(0)$.
\item[(iv)] Assume $a(0)=1$ and $a_\theta(0) = 0$, and denote $\lambda = \sqrt{k^2 - \frac{2k}{p/(p-2)+1}}$ .
The function $a = a(\theta)$ admits the parametrization 
\[
\begin{aligned}
& a = (t^2+\lambda^2)^{(k-1)/2}(t^2+k^2)^{-k/2}, \\
& \theta = \arctan\left(\frac{t}{k}\right) - \frac{k-1}{\lambda} \arctan\left(\frac{t}{\lambda}\right), 
\end{aligned}
\]
where $t \in \R$ and $\theta \in (-\frac{\pi}{2N}, \frac{\pi}{2N})$,
and for other values of $\theta$,
\[
a(\theta) = -a(\pi/N - \theta), \quad a(-\theta) = a(\theta).
\]
\item[(v)] The function $f = f(x,y)$ admits the parametrization 
\[
\begin{aligned}
& f = h^{k(2N-1)}\cos(N\tau), \\
& x = h^{2N-1}\bigl((k+\lambda)\cos\tau + (k-\lambda)\cos(2N-1)\tau\bigr), \\
& y = h^{2N-1}\bigl((k+\lambda)\sin\tau - (k-\lambda)\sin(2N-1)\tau\bigr), 
\end{aligned}
\]
where $\tau \in [0,2\pi], \ h > 0$.
\end{itemize}
\end{lemma}

The following Lemma is analogous to \cite[(3.13)]{Wolff:1984}.

\begin{lemma}
Let $1 < p < \infty$, $N \in \N$, and let $f(r,\theta) = r^ka(\theta)$ be
as in Lemma \ref{lem:tkachev}. The expression
\begin{equation*}
\left.\frac{d}{d\varepsilon}\right |_{\varepsilon = 0} \Delta_p (f+\varepsilon v) = 0
\end{equation*}
reduces formally to
\begin{equation*}
\operatorname{div}^\circ (A\nabla^\circ v) = 0,
\end{equation*}
where 
\begin{equation}
\label{A}
\begin{split}
& A(r,\theta) \\
& = r^{(p-2)(k-1)}(a_\theta^2 + k^2a^2)^\frac{p-4}{2}
\begin{pmatrix}
a_\theta^2 + (p-1)k^2a^2 & (p-2)kaa_\theta \cr
(p-2)kaa_\theta & k^2a^2 + (p-1)a_\theta^2 \cr
\end{pmatrix}.
\end{split}
\end{equation}
\end{lemma}
\begin{proof}
Since
\begin{equation*}
\begin{split}
& \frac{d}{d\varepsilon}\left(|\nabla^\circ f + \varepsilon \nabla^\circ v|^{p-2}(\nabla^\circ f + \varepsilon \nabla^\circ v)\right) = |\nabla^\circ f + \varepsilon\nabla^\circ v|^{p-2}\nabla^\circ v \\
& \qquad + (p-2)\bigl(\langle\nabla^\circ f , \nabla^\circ v\rangle + \varepsilon|\nabla^\circ v|^2\bigr)|\nabla^\circ f + \varepsilon\nabla^\circ v|^{p-4}(\nabla^\circ f + \varepsilon\nabla^\circ v),
\end{split}
\end{equation*}
the expression
\begin{equation*}
\left.\frac{d}{d\varepsilon}\right|_{\varepsilon = 0}\operatorname{div}^\circ \left(|\nabla^\circ f + \varepsilon \nabla^\circ v|^{p-2}(\nabla^\circ f + \varepsilon \nabla^\circ v)\right)
\end{equation*}
becomes
\begin{equation*}
\label{tensor_notation}
\begin{split}
& \operatorname{div}^\circ \left(|\nabla^\circ f|^{p-2}\nabla^\circ v + (p-2)|\nabla^\circ f|^{p-4}\langle\nabla^\circ f , \nabla^\circ v\rangle \nabla^\circ f\right) \\
& = \operatorname{div}^\circ \left(|\nabla^\circ f|^{p-2}\nabla^\circ v + (p-2)|\nabla^\circ f|^{p-4}(\nabla^\circ f\otimes \nabla^\circ f)\nabla^\circ v\right) \\
& = \operatorname{div}^\circ (A\nabla^\circ v),
\end{split}
\end{equation*}
where\footnote{
If $a = (a_1,\ldots,a_n)$ and $b = (b_1, \ldots, b_n)$ are vectors in $\R^n$,
their tensor product $a \otimes b$ is an $n\times n$ matrix with
\begin{equation*}
(a\otimes b)_{ij} = a_ib_j.
\end{equation*}
It is straightforward to check that the formula $\langle a, b\rangle a = (a \otimes a)b$
holds for vectors $a,b\in\R^n$.}
\begin{equation*}
A = |\nabla^\circ f|^{p-4}\left(|\nabla^\circ f|^2I + (p-2)(\nabla^\circ f\otimes\nabla^\circ f)\right).
\end{equation*}
Inserting
\begin{equation*}
|\nabla^\circ f| = r^{k-1}(a_\theta^2 + k^2a^2)^\frac{1}{2}
\end{equation*}
yields (\ref{A}).
\end{proof}

\begin{lemma}
\label{scaling_invariance}
Let $1<p<\infty$, $N\in\N$, and let $A$ be as in (\ref{A}). Then
\begin{equation*}
r^{(p-2)(k-1)}(a_\theta^2 + k^2a^2)^\frac{p-2}{2}|\xi|^2 \leq \langle A\xi , \xi \rangle \leq (p-1)r^{(p-2)(k-1)}(a_\theta^2 + k^2a^2)^\frac{p-2}{2}|\xi|^2
\end{equation*}
for all $\xi \in \R^2$. 
\end{lemma}
\begin{proof}
The eigenvalues $\mu$ of the matrix
\begin{equation*}
\widetilde{A} =
\begin{pmatrix}
a_\theta^2 + (p-1)k^2a^2 & (p-2)kaa_\theta \cr
(p-2)kaa_\theta & k^2a^2 + (p-1)a_\theta^2 \cr
\end{pmatrix}
\end{equation*}
are
\begin{equation*}
\begin{split}
\mu & = \frac{\operatorname{tr} \widetilde{A} \pm \sqrt{(\operatorname{tr} \widetilde{A})^2 - 4\det \widetilde{A}}}{2} \\
& = \frac{pa_\theta^2 + pk^2a^2 \pm \sqrt{p^2(a_\theta^2 + k^2a^2)^2 - 4\det \widetilde{A}}}{2}.
\end{split}
\end{equation*}
Since
\begin{equation*}
\det \widetilde{A} = (p-1)(a_\theta^2 + k^2a^2)^2,
\end{equation*}
we obtain
\begin{equation*}
\mu = \frac{\left(p\pm(p-2)\right)(a_\theta^2 + k^2a^2)}{2},
\end{equation*}
and the claim follows.
\end{proof}

{\em Remark.} Denoting $\lambda \leq \langle A\xi , \xi \rangle \leq \Lambda$, we have $\Lambda/\lambda = p-1$. Moreover, the eigenvectors
of $A$ work out to $\left(-a_\theta, \ ka \right)$ and $\left(ka, \ a_\theta \right)$. These observations
are not used in the present work.

\section{A weighted Sobolev space}
\label{section5}
Let $p$ and $N$ be fixed, and let $A$ be as in \eqref{A}. We will look for weak solutions $v$ to the equation
\begin{equation*}
\operatorname{div}^\circ (A\nabla^\circ v) = 0 \ \text{ in } \ \D^*
\end{equation*}
in the weighted $W^{1,2}$ space of $2\pi/N$-angular
periodic functions:
\begin{equation*}
Y_1 = \Bigl\{v \in W^{1,2}_\text{loc}(\D^*):
v\bigl(r,\theta+ \frac{2\pi}{N}\bigr) = v(r,\theta), \, \int_{\D}|v|^2r^{2\beta} + |\nabla^\circ  v|^2r^{2\alpha}\,dA < \infty\Bigr\}.
\end{equation*}
Here $\alpha = (p-2)(k-1)/2$ and $\beta$ is any number satisfying
$\alpha - 1 < \beta < 2\alpha - 1$. The inequality $\alpha - 1 < \beta$ will be needed for the imbedding of $Y_1$ to the weighted $L^2$ space
to be compact, and the
inequality $\beta < 2\alpha - 1$ for continuously differentiable functions to be dense in $Y_1$. The negativity of $\alpha-1$ is not
an issue since we assume $p>2$ and since we are ultimately interested in large values of the parameter $k$.

Define the weighted $L^2$ space as
\begin{equation*}
Y_0 = \left\{f\in L^2_\text{loc}(\D^*): f\bigl(r,\theta + \frac{2\pi}{N}\bigr)
= f(r,\theta), \, \int_\D |f(r,\theta)|^2r^{2\beta}\,dA < \infty\right\}.
\end{equation*}
The inner product in $Y_0$ is defined as
\begin{equation*}
\left(f\,|\,g \right)_{Y_0} = \int_\D f(r,\theta)r^\beta \, g(r,\theta)r^\beta \,dA,
\end{equation*}
and the inner products in $Y_1$, and in
\begin{equation*}
Y_0^* = \left\{f\in L^2_\text{loc}(\D^*): \int_\D |f(r,\theta)|^2r^{-2\beta}\,dA < \infty\right\} \\,
\end{equation*}
are defined accordingly. The dual pairing between $f \in Y_0$ and $g \in Y_0^*$ is
\begin{equation*}
\langle f | g \rangle = \int_\D fg\,dA.
\end{equation*}

%and the operator $T:Y_1 \to Y_0^*$ is
%\begin{equation}
%\label{operator_}
%Tu = -\operatorname{div}^\circ (A\,\nabla^\circ  u),
%\end{equation}
%where $A$ is the degenerate elliptic and symmetric matrix with entries
%in $C^\infty(\overline{\D} \setminus \{0\})$, see (\ref{A}).

In this section we prove three Lemmas about the spaces $Y_0$ and $Y_1$
that are omitted in the half-plane case of \cite{Wolff:1984}.
We prefer $\abs{\nabla f}$ over $\abs{\nabla^\circ f}$ in the notation,
because the expressions are equal and the moving frame will not become apparent until in section
\ref{section7}.

The first Lemma will be used in Lemma \ref{oblique_lemma}.

\begin{lemma}
The imbedding $id\colon Y_1 \to Y_0$ is compact.
\end{lemma}
\begin{proof}
Let $\varepsilon >0$ be small, and denote
$id = id_0 + id_1$, where $id_0$ and $id_1$ denote the restrictions of $id$
to functions in $Y_1$ defined on the annuli
\begin{equation*}
A_0 = \{(r,\theta)\colon 0<r<\varepsilon, 0\leq\theta < 2\pi\}
\end{equation*}
and
\begin{equation*}
A_1 = \{(r,\theta)\colon \varepsilon \leq r < 1, 0\leq\theta < 2\pi\},
\end{equation*}
respectively. The imbedding $id_1$ is compact for each $\varepsilon$,
because the imbedding $W^{1,2 } \to L^2$ is compact and we are away from
the origin. It suffices to show for $u\in Y_1$ that
\begin{equation*}
||u||^2_{Y_0(A_0)} := \int_{A_0}u^2r^{2\beta}\,dA \leq C_\varepsilon ||u||^2_{Y_1},
\end{equation*}
where $C_\varepsilon \to 0$ as $\varepsilon \to 0$. Since the imbeddings
$id_1$ are compact, this yields (see e.g. \cite[Thm. 0.34]{Folland:PDE}) that $id$ itself is compact.

Let $u\in Y_1$, $0<r<1/2$, and $0\leq\theta <2\pi$. Since
\begin{equation*}
u(r,\theta) \leq u(R,\theta) + \int_r^R |\nabla u(\rho,\theta)|\,d\rho
\end{equation*}
for each $R \in (r,1)$, we estimate
\begin{equation*}
|u(r,\theta)| \leq \vint_{(1/2,1)}|u(\rho,\theta)|\,d\rho + \int_r^1|\nabla u(\rho,\theta)|\,d\rho.
\end{equation*}
The Cauchy-Schwartz inequality yields
\begin{equation*}
\begin{split}
\int_r^1|\nabla u(\rho,\theta)|\,d\rho
&\leq \left(\int_r^1 |\nabla u(\rho,\theta)|^2 \rho^{2\alpha + 1}\,d\rho\right)^\frac{1}{2}
\left(\int_r^1 \rho^{-(2\alpha+1)}\,d\rho\right)^\frac{1}{2} \\
& \leq C_\alpha r^{-\alpha} \left(\int_r^1 |\nabla u(\rho,\theta)|^2 \rho^{2\alpha + 1}\,d\rho\right)^\frac{1}{2},
\end{split}
\end{equation*}
so we obtain for $u\in Y_1$ that
\begin{equation}
\label{former_latter}
\begin{split}
& \int_{A_0}u^2r^{2\beta}\,dA \\
& \leq C\int_{A_0}\left(\left(\vint_{(1/2,1)}|u(\rho,\theta)|\,d\rho\right)^2 + r^{-2\alpha}\int_r^1|\nabla u(\rho,\theta)|^2\rho^{2\alpha+1}\,d\rho \right)r^{2\beta}\,dA.
\end{split}
\end{equation}
The first term on the right-hand side of (\ref{former_latter}) is estimated using $r<1$ and the
Cauchy-Schwartz inequality:
\begin{equation*}
\begin{split}
& \int_{A_\varepsilon}\left(\vint_{(1/2,1)}|u(\rho,\theta)|\,d\rho\right)^2
r^{2\beta}\,dA \leq \int_{A_\varepsilon}\left(\vint_{(1/2,1)}|u(\rho,\theta)|
\,d\rho\right)^2\,dA\\
& \leq \int_{A_\varepsilon}\left(\int_{1/2}^1|u(\rho,\theta)|^2\rho^{2\beta}
\rho\,d\rho \right)\left( \int_{1/2}^1 \rho^{-2\beta}\,d\rho\right)\,dA
\leq C||u||_{Y_1}^2.
\end{split}
\end{equation*}
The second term on the right-hand side of (\ref{former_latter}) is estimated with
Fubini's theorem as
\begin{equation*}
\begin{split}
& \int_{A_\varepsilon} r^{-2\alpha}\int_r^1|\nabla u(\rho,\theta)|^2\rho^{2\alpha+1}\,d\rho \, r^{2\beta}\,dA \\
&= \int_0^\varepsilon r^{2\beta-2\alpha}\int_0^{2\pi}\int_r^1|\nabla u(\rho,\theta)|^2\rho^{2\alpha}\,\rho d\rho \,d\theta\,rdr \\
& \leq ||u||_{Y_1}^2\int_0^\varepsilon r^{2\beta-2\alpha+1}\,dr \leq \varepsilon^{2(\beta-\alpha+1)}||u||_{Y_1}^2.
\end{split}
\end{equation*}
Thus we obtain
\begin{equation*}
\begin{split}
& \int_{A_\varepsilon}u^2r^{2\beta}\,dA \leq C(\varepsilon^{2(\beta-\alpha+1)}+\varepsilon^2)||u||^2_{Y_1} \underset{\varepsilon \to 0}{\longrightarrow} 0
\end{split}
\end{equation*}
as wanted, since $\beta>\alpha-1$.
\end{proof}

The next Lemma is omitted in \cite[p. 390]{Wolff:1984} and is added here for completeness.
\begin{lemma}
\label{for_completeness}
Let $u\in Y_1$ be continuous.
For $k\in \N$ define $r_k = 2^{-k}$ and
\begin{equation*}
m_k = \min_\theta u(r_k,\theta).
\end{equation*}
Then
\begin{equation}
\label{easy}
\limsup_{k\to \infty} r_k^\alpha m_k \leq 0.
\end{equation}
\end{lemma}
\begin{proof}
Let us first assume that $u$ is a radial function. We will assume that
(\ref{easy}) is false and prove that
\begin{equation*}
\int_0^1\abs{u_r(r)}r^{2\alpha}\,rdr = \infty,
\end{equation*}
which contradicts the fact that $u\in Y_1$.

So let
\begin{equation*}
\limsup_{k\to \infty} r_k^\alpha u(r_k) = \varepsilon_0 > 0.
\end{equation*}
Then there exists a subsequence of $r_k$ (still call it $r_k$)
such that $r_k^\alpha u(r_k) \geq \varepsilon_0/2$ for each $k$. Fix $r_1$
from the subsequence, and choose $r_2$ from the subsequence small enough
such that $u(r_2) \geq 2u(r_1)$. We have, by the Fundamental Theorem
of Calculus (for Sobolev functions), that
\begin{equation*}
M := \int_{r_1}^{r_2}u_r(r)\,dr = u(r_2) - u(r_1) \geq u(r_2)/2 \geq r_2^{-\alpha}\varepsilon_0/4.
\end{equation*}
Let $v(r) = u_r(r)r^{\alpha+1/2}$, so that
\begin{equation*}
M = \int_{r_1}^{r_2}v(r)r^{-(\alpha+1/2)}\,dr = \langle v,r^{-(\alpha+1/2)}\rangle.
\end{equation*}
By elementary Hilbert space geometry, the smallest value of
\begin{equation*}
||v||^2_2 := ||v||^2_{L^2((r_1,r_2))} = \int_{r_1}^{r_2}v(r)^2\,dr
\end{equation*}
under the condition $\langle v,g\rangle = M$, is attained when
$v$ is parallel to $g$, i.e. $v = gM/||g||_2$. In our case,
$g(r) = r^{-(\alpha+1/2)}$ and
\begin{equation*}
\begin{split}
& \abs{\abs{v}}_2^2 \geq \frac{M^2}{||g||_2^2} = \frac{M^2}
{\int_{r_1}^{r_2}r^{-(2\alpha+1)}dr} = \frac{M^2}{r_2^{-2\alpha} - r_1^{-2\alpha}} \\
&  \geq \frac{\varepsilon_0^2}{16}\cdot \frac{r_2^{-2\alpha}}{r_2^{-2\alpha}-r_1^{-2\alpha}} \
= \frac{\varepsilon_0^2}{16}\cdot \frac{1}{1-\left(\frac{r_2}{r_1}\right)^{2\alpha}}
\geq C_\alpha \varepsilon_0^2,
\end{split}
\end{equation*}
since $r_2/r_1 \leq 1/2$ by definition.

Now choose $r_3$ from the subsequence such that $u(r_3) \geq 2u(r_2)$
and repeat the process above. Continuing in a similar fashion and summing over the
chosen radii $r_j$, we have
\begin{equation*}
||u||_{Y_1}^2 \geq \sum_{j=1}^\infty||v||^2_{L^2_{((r_j,r_{j+1}))}} = \infty,
\end{equation*}
which concludes the radial case.

When $u$ is not radial, denote
\begin{equation*}
\widetilde{u}(r) = \frac{1}{2\pi}\int_0^{2\pi} u(r,\theta)\,d\theta
\end{equation*}
and repeat the proof above as follows: Choose $r_2$
from the subsequence such that 
\begin{equation*}
\min_\theta u(r_2,\theta) \geq 2\max_\theta u(r_1,\theta).
\end{equation*}
Then
\begin{equation*}
M := \int_{r_1}^{r_2}\widetilde{u}_r(r)\,dr \geq \frac{1}{2}\min_\theta u(r_2,\theta)
\geq r_2^{-\alpha}\varepsilon_0/4.
\end{equation*}
As before, we obtain
\begin{equation*}
\int_{r_1}^{r_2}\abs{\widetilde{u}_r(r)}^2r^{2\alpha}\,rdr \geq C_\alpha\varepsilon_0^2.
\end{equation*}
Finally, since
\begin{equation*}
\abs{\widetilde{u}_r(r)}^2 \leq \frac{1}{2\pi}\int_0^{2\pi}\abs{u_r(r,\theta)}^2\,d\theta,
\end{equation*}
we have, by Fubini's theorem, that
\begin{equation*}
\begin{split}
& \int_\D|\nabla u|^2 r^{2\alpha}\,dA
\geq \int_0^{2\pi}\int_0^1|u_r(r,\theta)|^2\,r^{2\alpha}\,r\,dr\,d\theta \\
& = \int_0^1 \int_0^{2\pi}|u_r(r,\theta)|^2\,d\theta\,r^{2\alpha+1}\,dr
\geq 2\pi \int_0^1 |\widetilde{u}_r(r)|^2\,r^{2\alpha+1}\,dr \\
& \geq 2\pi \sum_{j=1}^\infty\int_{r_j}^{r_{j+1}}
|\widetilde{u}_r(r)|^2\,r^{2\alpha+1}\,dr = \infty,
\end{split}
\end{equation*}
contradicting $u\in Y_1$.
\end{proof}

The result below is used in the proof of Lemma \ref{stokes_thm}.
\begin{lemma}
\label{dense}
The space $C^1(\D)$ of continuously differentiable functions in $\D$
is dense in $Y_1$, i.e. $Y_1$ is the closure of $C^1(\D)$ under the norm
\begin{equation*}
||f||_{Y_1} = \left(\int_\D |f(r,\theta)|^2r^{2\beta} + |\nabla f(r,\theta)|^2r^{2\alpha }\,dA\right)^\frac{1}{2}.
\end{equation*}
\end{lemma}
\begin{proof}
Let $u \in Y_1$ and $\varepsilon > 0$. We are looking for a function
$v = v_\varepsilon \in C^1(\D)$ such that
$||u - v||_{Y_1} \to 0$ when $\varepsilon \to 0$. By truncating $u$
we may assume that $u$ is bounded. We may also assume that $u$ is
radial; the general case follows as in the proof of Lemma \ref{for_completeness}.

Let $\varphi = \varphi_\varepsilon \in C^\infty(0,1)$ be such that
$\varphi(r) = 0$ when $0<r< \varepsilon$, $\varphi(r) = 1$ when
$2\varepsilon < r < 1$, and $0 \leq \varphi_r \leq C/\varepsilon$.
We will show below that $u\varphi \in Y_1$ satisfies
$||u-u\varphi||_{Y_1} \to 0$ as $\varepsilon \to 0$. Thereafter the
desired function $v$ is obtained from $u\varphi$ by a standard
convolution approximation left to the reader.

We start with
\begin{equation*}
\begin{split}
& ||u - u\varphi||_{Y_1}^2 = ||u(1-\varphi)||_{Y_1}^2 \\
& = \int_0^1 |u(1-\varphi)|^2\,r^{2\beta+1}\,dr
 + \int_0^1 |\bigl(u(1-\varphi)\bigr)_r|^2\,r^{2\alpha+1}\,dr.
\end{split}
\end{equation*}
The first integral is bounded above by
\begin{equation*}
C\int_0^{2\varepsilon}|u|^2\,r^{2\beta+1}\,dr
\end{equation*}
and goes to zero with $\varepsilon$, since $u \in Y_0$. For the second integral
we estimate
\begin{equation*}
\begin{split}
& \int_\varepsilon^{2\varepsilon} |\bigl(u(1-\varphi)\bigr)_r|^2\,r^{2\alpha+1}\,dr \\
& \leq 2 \int_\varepsilon^{2\varepsilon} |1-\varphi|^2|u_r|^2\,r^{2\alpha+1}\,dr
+ 2 \int_\varepsilon^{2\varepsilon} |u|^2|(1-\varphi)_r|^2\,r^{2\alpha+1}\,dr.
\end{split}
\end{equation*}
Here the first integral goes to zero with $\varepsilon$, since $u\in Y_1$.
Since $|(1-\varphi)_r| = |\varphi_r| \leq C/\varepsilon$, the second
second integral is estimated using the Cauchy-Schwartz inequality:
\begin{equation*}
\begin{split}
& \frac{C}{\varepsilon^2}\int_\varepsilon^{2\varepsilon}|u|^2\,r^{2\alpha+1}\,dr 
 \leq \frac{C}{\varepsilon^2}\left(\int_\varepsilon^{2\varepsilon}|u|^2\,r^{2\beta+1}\,dr\right)^\frac{1}{2} 
\left(\int_\varepsilon^{2\varepsilon}|u|^2\,r^{4\alpha - 2\beta+1}\,dr\right)^\frac{1}{2} \\
& \leq \frac{C}{\varepsilon^2}||u||_{Y_0}\left(\int_\varepsilon^{2\varepsilon}|u|^2\,r^{4\alpha - 2\beta+1}\,dr\right)^\frac{1}{2} 
\leq \frac{C}{\varepsilon^2}||u||_{Y_0}||u||_\infty\left(\int_\varepsilon^{2\varepsilon}r^{4\alpha - 2\beta+1}\,dr\right)^\frac{1}{2} \\
& = C \varepsilon^{2\alpha - \beta -1},
\end{split}
\end{equation*}
which goes to zero when $\beta < 2\alpha - 1$.
\end{proof}

\section{Regularity}
\label{section6}

We proceed to prove a priori regularity of solutions to
the linearized $p$-Laplace equation \eqref{linearized_plap}.
Since the coefficients of the matrix $A$ in \eqref{A} are in the class $C^\infty(\overline{\D}\setminus\{0\})$, also the solutions
are in this class by standard linear regularity theory. The question of interest
is regularity at the origin. 

\subsection{The half-plane case}
We start by quoting some results from \cite[pp. 387--390]{Wolff:1984}.
Fix arbitrary constants $\lambda>0$, $\beta > 0$ and $0 < \alpha < \beta$,
denote $S^{\lambda} = [0,\lambda) \times (0,\infty) \subset \R^2_+$, and consider the Hilbert space $\widetilde{Y}_1$
defined via
\[
\widetilde{Y}_0 = \left\{ f \in L^2_{\mathrm{loc}}(\R^2_+) : f(x+\lambda,y) = f(x,y), \ \int_{S^\lambda} |f(x,y)|^2e^{-2\beta y} \ dxdy < \infty \right\},
\]
\[
\widetilde{Y}_1 = \widetilde{Y}_0 \cap \left\{ f \in L^2_{\mathrm{loc}}(\R^2_+) : \nabla f \in L^2_{\mathrm{loc}}(\R^2_+), \ \int_{S^\lambda} | \nabla f(x,y)|^2e^{-2\alpha y} \ dxdy < \infty \right\}.
\]
Fix $\widetilde{A} \colon \overline{\R^2_+} \to 2\times 2$ real symmetric matrices. Assume $\widetilde{A}$ is $C^\infty$ on $\overline{\R^2_+}$,
$\widetilde{A}(x+\lambda,y) = \widetilde{A}(x,y)$ and that there exists a constant $C > 0$ such that
\begin{equation}
\label{Wolff_class}
C^{-1}e^{-2\alpha y}|\xi|^2 \le \langle \widetilde{A}(x,y) \xi , \xi \rangle \le Ce^{-2\alpha y} |\xi|^2.
\end{equation}
Then the following two results hold:
\begin{thm}[\cite{Wolff:1984}, Lemma 3.8]
\label{wolff_thm1}
If $u \in \widetilde{Y}_1 \cap C^\infty(\overline{\R^2_+})$ satisfies $\diver(\widetilde{A}\nabla u) = 0$ in $\R^2_+$, then $u$ is bounded. In fact,
$u(x,y) \le \max_t u(t,0)$ for all $(x,y) \in \R^2_+$.
\end{thm}
\begin{thm}[\cite{Wolff:1984}, Lemma 3.12]
\label{wolff_thm2}
If $u \in \widetilde{Y}_1 \cap C^\infty(\overline{\R^2_+})$ satisfies $\diver(\widetilde{A}\nabla u) = 0$ in $\R^2_+$, then there are $\gamma > 0$
and $\mu \in \R$ such that $|u(x,y) - \mu| \le 2e^{-\gamma y}\|u\|_\infty$ for all $(x,y) \in \R^2_+$. Consequently
$\nabla u \in L^q(S^\lambda)$ for all $q \in (0,\infty]$.
\end{thm}

\subsection{The disk case}
The weak form of
$\diver(\widetilde{A}\nabla u) = 0$ in $\R^2_+$, where $\widetilde{A}$ satisfies \eqref{Wolff_class},
is the Euler-Lagrange equation for minimizing
\[
\int_{S^\lambda} \langle \widetilde{A}\nabla u, \nabla u \rangle \ dxdy % \sim \int_{S^\lambda} |\nabla u|^2 e^{-2\alpha y} \ dxdy
\]
among functions in $\widetilde{Y}_1$. Analogously the weak form of the equation $\operatorname{div}^\circ (A\nabla^\circ v) = 0$ in $\D^*$,
where $A$ satisfies \eqref{A},
is the Euler-Lagrange equation for minimizing
\begin{equation}
\label{functional}
\int_{\D^*} \langle A\nabla^\circ v, \nabla^\circ v \rangle \ rdrd\theta
\end{equation}
among functions in $Y_1$.

\begin{thm}
Let $v \in Y_1$ minimize \eqref{functional} with a given boundary data.
Map the strip $S = \{(x,y)\in \R^2_+ \colon -\pi < x \le \pi\}$ to $\D^*$ via the map
$G \colon \C \to \C$, $G(z) = e^{iz}$. Then the composed function $u = v \circ G$ minimizes
\begin{equation}
\label{eq:Wolff_functional}
\int_{S}c(x,y)\abs{\nabla u(x,y)}^2e^{-2\alpha y}\,dx\,dy,
\end{equation}
where $\alpha = \big((p-2)(k-1)+2\big)/2$ and $C^{-1} \le c(x,y) \le C$ for some $C > 0$.
\end{thm}
\begin{proof}
The expression \eqref{functional} has the form
\begin{equation}
\label{expression2}
\int_{\D^*}(a_\theta^2 + k^2a^2)^\frac{p-2}{2} r^{(p-2)(k-1)}\abs{\nabla v(r,\theta)}^2\,r\,dr\,d\theta.
\end{equation}
In Cartesian coordinates $(\widetilde{x},\widetilde{y}) \in \D$, the expression \eqref{expression2} reads
\[
\int_{\D^*}c(\widetilde{x},\widetilde{y})(\widetilde{x}^2+\widetilde{y}^2)^{(p-2)(k-1)/2}\abs{\nabla v(\widetilde{x},\widetilde{y})}^2(\widetilde{x}^2+\widetilde{y}^2)^{1/2}\,d\widetilde{x}\,d\widetilde{y},
\]
where $c(\widetilde{x},\widetilde{y}) = (a_\theta^2 + k^2a^2)^\frac{p-2}{2}$.
A change of variables $(\widetilde{x},\widetilde{y}) = G(x,y)$ yields
\begin{equation}
\label{map_conclusion}
\begin{split}
& \int_{\D^*}(a_\theta^2 + k^2a^2)^\frac{p-2}{2} r^{(p-2)(k-1)}\abs{\nabla v(r,\theta)}^2\,r\,dr\,d\theta \\
& = \int_S {c}(x,y) \abs{\nabla u(x,y)}^2e^{-2\alpha y}\,dx\,dy,
\end{split}
\end{equation}
because $r^2 = \widetilde{x}^2+\widetilde{y}^2 = e^{-2y}$ and because the Jacobian of $G$
is $r = e^{-y}$.
A minimizer of \eqref{functional} minimizes the left-hand side in \eqref{map_conclusion}; hence also
the right-hand side. Since $c(x,y) = a_\theta^2 + k^2a^2$ is bounded away from zero and infinity,
the claim follows.
\end{proof}

\begin{thm}
\label{Wolff_disk}
If $u \in Y_1$ satisfies $\operatorname{div}^\circ (A\nabla^\circ u) = 0$ in $\D^*$, then $u \in L^\infty$ and $u(r,\theta) \le \max_\theta u(1,\theta)$ in $\D^*$. Moreover, the expression $\sqrt{x^2+y^2}\ \nabla u(x,y)$ stays bounded in $\D^*$.
\end{thm}
\begin{proof}
Consider a solution $v$ to the linearized $p$-Laplacian \eqref{linearized_plap} in $\D^*$. When $S$ is mapped to $\D^*$ via $G$,
the function $u = v\circ G$ on $S$ minimizes a quadratic functional that belongs to the class \eqref{Wolff_class}. 
By Theorems \ref{wolff_thm1} and \ref{wolff_thm2}, both $u$ and $\nabla u$ stay bounded in $S$.
\end{proof}

\section{The oblique derivative problem}
\label{section7}

The main result of this paper is the following that
corresponds to \cite[Lemma 3.15]{Wolff:1984}.
\begin{thm}
\label{Neumann}
Let $p>2$, let $A$ be as in (\ref{A}), and let $M>0$. There exists a solution
$v \in Y_1$ to 
\begin{equation*}
Tv := -\operatorname{div}^\circ (A\nabla^\circ v) = 0 \quad \text{ in } \D^*
\end{equation*}
such that
\begin{equation*}
\int_0^{2\pi} \frac{dv}{dn}(1,\theta)d\theta = M,
\end{equation*}
where $n$ denotes the outer normal vector on $\partial\D$.
\end{thm}
We prove Theorem \ref{Neumann} via a series of Lemmas. The first step is to transform the problem to an oblique derivative problem.
\begin{lemma}
\label{Oblique}
Denote $n^* = An$. Assume that there exists a function $\psi \colon \partial \D \to \R$ and
a solution $v \in Y_1$ to
\begin{equation}
\label{oblique}
\begin{cases}
Tv &= 0 \qquad \text{ in } \D^* \\
\frac{\partial v}{\partial n^*} + \tau\frac{\partial v}{\partial \theta} &= \frac{\psi}{q} \qquad \text{on } \partial\D,
\end{cases}
\end{equation}
such that
\begin{equation}
\label{constraint}
\int_0^{2\pi}\psi(\theta)d\theta = M.
\end{equation}
Then Theorem \ref{Neumann} holds.
\end{lemma}
\begin{proof}
In our moving frame the outer normal is $n = (1,0)$, and the conormal on $\partial\D$ with
respect to $T$ is
\begin{equation*}
n^*(\theta) := An = A(1,\theta)\begin{pmatrix} 1 \\ 0 \end{pmatrix} = (a^2 + k^2a_\theta^2)^\frac{p-4}{2}\begin{pmatrix}a_\theta^2 + (p-1)k^2a^2 \\ (p-2)kaa_\theta \end{pmatrix}.
\end{equation*}
With
\begin{equation*}
\omega = (a_\theta^2 + k^2a^2)^\frac{p-4}{2},
\end{equation*}
we have
\begin{equation*}
n^* = \omega\left(a_\theta^2 + (p-1)k^2a^2\right)
\begin{pmatrix}
1\\
0
\end{pmatrix}
+ \omega(p-2)kaa_\theta
\begin{pmatrix}
0 \\
1
\end{pmatrix},
\end{equation*}
so that
\begin{equation*}
\begin{split}
n & = 
\begin{pmatrix}
1 \\
0
\end{pmatrix}
=
\frac{1}{\omega (a_\theta^2 + (p-1)k^2a^2)}
\left(n^* - (p-2)\omega kaa_\theta
\begin{pmatrix}
0 \\
1
\end{pmatrix} \right) \\
\; \\
& = q\left(n^* + \tau \begin{pmatrix}0\\1\end{pmatrix}\right),
\end{split}
\end{equation*}
where
\begin{equation}
\label{wolff_3.4}
\begin{split}
q(\theta) & = \frac{(a_\theta^2 + k^2a^2)^\frac{4-p}{2}}{a_\theta^2 + (p-1)k^2a^2}, \\
\tau(\theta) &= -(a_\theta^2 + k^2a^2)^\frac{p-4}{2}(p-2)kaa_\theta.
\end{split}
\end{equation}
It follows that
\begin{equation}
\label{partial_n}
\frac{\partial }{\partial n} = q\left(\frac{\partial }{\partial n^*} + \tau\frac{\partial }{\partial \theta}\right),
\end{equation}
and in particular that the equation 
\begin{equation*}
\frac{\partial v}{\partial n} = \psi
\end{equation*}
for some $\psi\colon \partial\D \to \R$, is equivalent to the equation
\begin{equation}
\label{conormal_dual}
\frac{\partial v}{\partial n^*} + \tau\frac{\partial v}{\partial \theta} = \frac{\psi}{q}.
\end{equation}
\end{proof}

The following corresponds to \cite[Lemma 3.7]{Wolff:1984}.
\begin{lemma}
\label{oblique_lemma}
Let $\psi$ be a function on $\partial \D$, and let $q$ and $\tau$ be as in
(\ref{wolff_3.4}). Let $E$ be the set of all admissible boundary values
$f(\theta) = F(1,\theta)$ of solutions $F \in Y_1$ to 
\begin{equation}
\label{dual_problem}
\begin{cases}
TF &= 0 \, \text{ in } \D^* \\
\frac{\partial F}{\partial n^*} - \frac{\partial}{\partial
\theta}\left(\tau F\right) &= 0 \, \text{ on } \partial\D.
\end{cases}
\end{equation}
Then $E$ is finite-dimensional, and the oblique derivative problem
(\ref{oblique}) has a solution $v \in Y_1$ if $\frac{\psi}{q} \perp E$,
i.e. if
\begin{equation}
\label{finitedim}
\int_0^{2\pi} \frac{\psi(\theta)}{q(\theta)}f(\theta)\,d\theta = 0 \quad \text{ for
all } f \in E.
\end{equation}
\end{lemma}
\begin{proof}
The strategy is to first consider the
problem of finding a function $u\in Y_1$ such that
\begin{equation}
\label{homogeneous}
\begin{cases}
Tu &= g \qquad \text{ in } \D \\
\frac{\partial u}{\partial n^*} + \tau\frac{\partial u}{\partial \theta} &= 0 \qquad \text{ on } \partial\D,
\end{cases}
\end{equation}
i.e. to find a suitable condition for $g \in Y_1^*$ such that the problem
(\ref{homogeneous})
admits a solution. Thereafter the problem (\ref{oblique}) is reduced to the
problem (\ref{homogeneous}).

Following \cite[Chap. 7]{Folland:PDE}, we start by constructing a suitable Dirichlet form.\footnote{The notation in
\cite{Folland:PDE} is different from the notation in
\cite[Lemma 3.7]{Wolff:1984}; our notation corresponds to that in
\cite{Folland:PDE}.} 
Our form
$D: Y_1 \times Y_1 \to \R$ should satisfy
\begin{equation}
\label{dirichletform}
D(v,u) - \langle v\,|\,Tu\rangle
= \int_{\partial\D}v\left(\frac{\partial u}{\partial n^*} + \tau\frac{\partial u}{\partial \theta}\right)d\theta,
\end{equation}
so that the condition 
\begin{equation*}
D(v,u) = \langle v\,|\,g\rangle \quad \text{ for all } \, v \in Y_1
\end{equation*}
guarantees that $u \in Y_1$ is a weak solution to (\ref{homogeneous}).

Let $B = A + \mathcal{C}$, where
\begin{equation}
\label{C}
\mathcal{C} =
\left(\begin{array}{rr}
0 & -c\\
c & 0
\end{array}\right),
\end{equation}
and $c \in C^\infty(\overline{\D})$ is any function such that
$-c(1,\theta) = \tau(\theta)$ and $c(r,\theta) = 0$ for $r < 1/2$.
Our Dirichlet form is defined such that
\begin{equation}
\label{form_definition}
\int_{\partial\D}v\frac{\partial u}{\partial n_B^*}d\theta
= \int_\D v\operatorname{div}^\circ (A\nabla^\circ u) \,dA + D(v,u),
\end{equation}
i.e.
\begin{equation*}
D(v,u) 
 = \int_\D\langle \nabla^\circ v, B\nabla^\circ u\rangle \,dA 
+ \int_\D v\bigl\{e_\theta(c_{21})e_r(u)
 + e_r(c_{12})e_\theta (u) \bigr\}\,dA;
\end{equation*}
see subsection \ref{on_dirichlet} (especially Lemma \ref{the_form}) below.
Because our Dirichlet form is coercive and the injection
$Y_1 \to Y_0$ is compact, the standard Fredholm-Riesz-Schauder theory
(\cite[Thm. 7.21]{Folland:PDE}) yields that the space
\begin{equation*}
\begin{aligned}
\mathcal{W} & = \{u \in Y_1: D^*(v,u) = 0 \text{ for each } v\in Y_1\} \\ & = \{v \in Y_1: D(v,u) = 0 \text{ for each } u\in Y_1\}
\end{aligned}
\end{equation*}
is finite-dimensional in $Y_0$, and that the problem
(\ref{homogeneous})
admits a solution whenever
\begin{equation*}
\langle g \,|\, v\rangle = \int_\D gv\,dA = 0
\end{equation*}
for each $v \in \mathcal{W}$.
The remaining step is to solve the problem (\ref{oblique}) with the additional condition (\ref{finitedim}).

Let $\psi\colon\partial\D\to\R$ be continuous and such that
\begin{equation}
\label{h_perp_E}
\int_0^{2\pi} \frac{\psi(\theta)}{q(\theta)}v(1,\theta)d\theta = 0
\end{equation}
holds for each $v \in \mathcal{W}$, and let $h \in Y_1$ be such that
\begin{equation*}
\frac{\partial h}{\partial n^*}(1,\theta) + \tau(\theta)\frac{\partial h}{\partial \theta}(1,\theta) = \frac{\psi(\theta)}{q(\theta)}.
\end{equation*}
We claim that
\begin{equation}
\label{TH_perp_W}
\int_{\D}v\,Th\,dA = 0 \qquad \text{ for all } v\in \mathcal{W}.
\end{equation}
Indeed, let $v \in \mathcal{W}$, i.e. $D(v,u) = 0$ for each $u \in Y_1$. By (\ref{dirichletform}),
\begin{equation*}
D(v,h) - \int_{\D}vTh\,dA = \int_0^{2\pi} v\left(\frac{\partial h}{\partial n^*} + \tau\frac{\partial h}{\partial \theta}\right)\,d\theta,
\end{equation*}
so (\ref{TH_perp_W}) follows by (\ref{h_perp_E}).
Since (\ref{TH_perp_W}) holds, we can solve the problem (\ref{homogeneous}) with
$g = -Th$, obtaining a weak solution $u$ to
\begin{equation*}
\begin{cases}
Tu &= -Th \quad \text{ in } \D \\
\frac{\partial u}{\partial n^*} + \tau\frac{\partial u}{\partial \theta} &= 0 \quad\qquad \text{on } \partial\D,
\end{cases}
\end{equation*}
and the function $w = u + h$ solves (\ref{oblique}).
\end{proof}

\begin{rem}
Any solution $v \in Y_1$ to \eqref{oblique} has angular period $2\pi/N$,
because we have chosen to include the periodicity in the space $Y_1$.
Without the inclusion, the periodicity of $v$ would follow from
a periodic boundary function $\psi$ below.
However, the argument below is non-constructive, and constructing a
quantitative solution of the Neumann problem with
specific values of $N$ is beyond the scope of this paper.
\end{rem}

{\bf Proof of Theorem \ref{Neumann}.} By Lemma \ref{oblique_lemma}, what we
need is
\begin{equation}
\label{thisreads}
\int_0^{2\pi} \frac{\psi}{q}\,g\, d\theta = 0 \quad \text{ for all } g\in E, \quad \text{ and } \quad \int_0^{2\pi} \psi\,d\theta = M.
\end{equation}
Writing $\varphi = \psi/q$ and using the bracket notation,
(\ref{thisreads}) reads
\begin{equation*}
\langle\varphi,g\rangle = 0 \quad \text{ for all } g\in E \quad \text{ and } \quad \langle\varphi,q\rangle = M.
\end{equation*}
A necessary condition clearly is $q \notin E$, but it is also sufficient: if $q \notin E$, we write $q = q_E + q_\perp$, where $q_E \in E$ and
$q_\perp \in E^\perp$. Then
\begin{equation*}
\langle \varphi, q \rangle = \langle \varphi, q_\perp \rangle,
\end{equation*}
so if $q \notin E$, we can choose any function $\psi$ such that $\varphi = \frac{\psi}{q} \notin E$ in order to have $\langle \varphi, q \rangle \neq 0$,
and then multiply by a constant to obtain $\langle \varphi, q\rangle = M$.

In order to finish the proof, we need to show that $q \notin E$.
Suppose on the contrary that $q\in E$, i.e. $q(\theta) = F(1,\theta)$ for some solution
$F \in Y_1\cap C^\infty(\overline{\D}\setminus\{0\})$ to
\begin{equation*}
\begin{cases}
TF &= 0 \qquad \qquad\qquad \text{ in } \D \\
\frac{\partial F}{\partial n^*}(1,\theta) &= \frac{d}{d\theta}\left(\tau(\theta)q(\theta)\right) \quad \text{on } \partial\D.
\end{cases}
\end{equation*}
\label{blahblah}
Let $\theta_0$ be a global minimum point of $q$. (Such a point exists since
$q \in C^\infty(\partial\D)$.) By Lemma \ref{Wolff_disk}, $(1,\theta_0)$ is
a minimum point of $F$ on $\overline{\D}$. By (\ref{partial_n}),
\begin{equation*}
\frac{\partial F}{\partial n^*}(1,\theta) = \frac{1}{q(\theta)}\frac{\partial F}{\partial n}(1,\theta) - \tau(\theta)\frac{\partial F}{\partial \theta}(1,\theta).
\end{equation*}
At the minimum point $(1,\theta_0)$, the last term on the right-hand side equals zero, and the outer normal derivative of $F$ has to be nonpositive.
Since $q > 0$, we obtain
\begin{equation*}
\frac{\partial F}{\partial n^*}(1,\theta_0) \leq 0,
\end{equation*}
i.e.
\begin{equation*}
\left.\frac{d}{d\theta}\right|_{\theta = \theta_0}\bigl(\tau(\theta)q(\theta)\bigr) \leq 0,
\end{equation*}
by (\ref{dual_problem}).
But this is impossible by the following essential result that
corresponds to \cite[Lemma 3.5]{Wolff:1984} and that finishes the proof.
\hfill$\Box$

\begin{lemma}
\label{calculation}
Let $p>2$, $k \geq 2$, let $q$ and $\tau$ be as in (\ref{wolff_3.4}), and let $\theta_0$ be a minimum point of $q$. Then
\begin{equation*}
\left.\frac{d}{d\theta}\right|_{\theta = \theta_0}\bigl(\tau(\theta)q(\theta)\bigr) > 0.
\end{equation*}
\end{lemma}
\begin{proof}
Recall that
\begin{equation*}
\begin{split}
q & = \left(a_\theta^2 + k^2a^2\right)^\frac{4-p}{2}\left(a_\theta^2 + (p-1)k^2a^2\right)^{-1}, \\
\tau q &= -\left(a_\theta^2 + (p-1)k^2a^2\right)^{-1}(p-2)kaa_\theta,
\end{split}
\end{equation*}
and that $a$ satisfies $a_{\theta\theta} = -Va$, where
\begin{equation*}
V = \frac{\bigl((2p-3)k^2-(p-2)k\bigr)a_\theta^2+\bigl((p-1)k^2-(p-2)k\bigr)k^2a^2}{(p-1)a_\theta^2 + k^2a^2}.
\end{equation*}

We start with the simpler case $p=4$, where
\begin{equation*}
\begin{split}
q & = (a_\theta^2 + 3k^2a^2)^{-1}, \\
\tau q &= -(a_\theta^2 + 3k^2a^2)^{-1}2kaa_\theta, \\
V & = \frac{(5k^2-2k)a_\theta^2 + (3k^2-2k)k^2a^2}{3a_\theta^2 + k^2a^2}.
\end{split}
\end{equation*}
Now
\begin{equation*}
\begin{split}
q_\theta &= -(a_\theta^2 + 3k^2a^2)^{-2}2(a_\theta a_{\theta\theta} + 3k^2aa_\theta) \\
&= -2aa_\theta(a_\theta^2 + 3k^2a^2)^{-2}(3k^2 - V),
\end{split}
\end{equation*}
which is zero only when $a=0$ or $a_\theta = 0$ or $V = 3k^2$.
The last alternative reads
\begin{equation*}
(5k^2-2k)a_\theta^2 + (3k^2-2k)k^2a^2 = 3k^2(3a_\theta^2 + k^2a^2)
\end{equation*}
and simplifies to
\begin{equation*}
a_\theta^2(4k^2 + 2k) = k^2a^2(-2k),
\end{equation*}
which is impossible. Next, consider the case $a=0$. Denote
$A = 2(a_\theta^2 + 3k^2a^2)^{-2}$, so that
$q_\theta = -Aaa_\theta(3k^2-V)$, and
\begin{equation*}
q_{\theta\theta} = \bigl(-A(3k^2-V)\bigr)_\theta aa_\theta -A(3k^2-V)(aa_\theta)_\theta.
\end{equation*}
When $a = 0$, this equals
\begin{equation*}
-a_\theta^2A(3k^2-V),
\end{equation*}
and thus has the same sign as $V-3k^2$. But when $a=0$, we have
$V = (5k^2-2k)/3$, and $3k^2-V > 0$. Hence $q_{\theta\theta} < 0$ when $a=0$,
i.e. points where $a=0$ are local maxima for $q$. Hence a local minimum of
$q$ can occur only when $a_\theta = 0$. At such a point,
since $a_{\theta\theta}=-Va$,
\begin{equation*}
\begin{split}
(\tau q)_\theta & = -(a_\theta^2+3k^2a^2)^{-2}\bigl(2k(a_\theta^2+3k^2a^2)(aa_\theta)_\theta - 2kaa_\theta(a_\theta^2+3k^2a^2)_\theta \bigr) \\
&= -(3k^2a^2)^{-2}3k^2a^2 \cdot 2kaa_{\theta\theta} = \frac{2V}{3k} >0.
\end{split}
\end{equation*}

Now consider the case $p\neq 4$. First, we consider the sign of
$(\tau q)_\theta$. With $B:= a_\theta^2 + (p-1)k^2a^2 >0$, we have
$\tau q = -(p-2)kB^{-1}aa_\theta$. Disregarding $(p-2)k>0$, the sign of
$(\tau q)_\theta$ is the same as the sign of
\begin{equation*}
B^{-2}(B_\theta aa_\theta - B(aa_\theta)_\theta).
\end{equation*}
We disregard $B^{-2}>0$, and since $(aa_\theta)_\theta = a_\theta^2 - Va^2$,
we have
\begin{equation*}
\sgn\bigl((\tau q)_\theta\bigr) = \sgn\bigl( B_\theta aa_\theta - B(a_\theta^2 - Va^2) \bigr).
\end{equation*}
Next, we calculate
\begin{equation}
\label{B_theta}
B_\theta = 2a_\theta a_{\theta\theta} + 2(p-1)k^2aa_\theta = 2aa_\theta\bigl((p-1)k^2-V\bigr),
\end{equation}
so that
\begin{equation*}
\sgn\bigl((\tau q)_\theta\bigr) = \sgn\bigl( 2a^2a_\theta^2\bigl((p-1)k^2-V\bigr) - B(a_\theta^2 - Va^2) \bigr).
\end{equation*}
Inserting $B$ yields that $(\tau_q)_\theta$ has the same sign as the expression
\begin{equation*}
2a^2a_\theta^2\bigl((p-1)k^2-V\bigr) - (a_\theta^2+(p-1)k^2a^2)(a_\theta^2 - Va^2),
\end{equation*}
which simplifies to
\begin{equation*}
(p-1)k^2Va^4 + \bigl((p-1)k^2-V\bigr)a^2a_\theta^2-a_\theta^4,
\end{equation*}
and factorizes to
\begin{equation*}
\left((p-1)k^2a^2 - a_\theta^2\right)\left(Va^2 + a_\theta^2\right).
\end{equation*}
Thus we conclude that $(\tau q)_\theta$ is positive only when
$a_\theta^2 < (p-1)k^2a^2$. 

Next consider $q_\theta$. With $A:=a_\theta^2 + k^2a^2$ and again
$B:= a_\theta^2 + (p-1)k^2a^2$, we have $q = A^{(4-p)/2}B^{-1}$ and
\begin{equation}
\label{q_form}
q_\theta = \frac{4-p}{2}A^\frac{2-p}{2}A_\theta B^{-1} - A^\frac{4-p}{2}B^{-2}B_\theta
= B^{-2}A^\frac{2-p}{2}\left(\frac{4-p}{2}A_\theta B - AB_\theta\right).
\end{equation}
We already calculated $B_\theta$ in (\ref{B_theta}), and similarly
$A_\theta = 2(k^2 - V)aa_\theta$,
so in (\ref{q_form}),
\begin{equation*}
\begin{split}
& \frac{4-p}{2}A_\theta B - AB_\theta \\
&= (4-p)(k^2-V)aa_\theta(a_\theta^2+(p-1)k^2a^2) - 2(a_\theta^2 + k^2a^2)\bigl((p-1)k^2-V\bigr)aa_\theta \\
&= aa_\theta\left[(4-p)(k^2-V)(a_\theta^2+(p-1)k^2a^2)-2\bigl((p-1)k^2-V\bigr)(a_\theta^2+k^2a^2)\right] \\
& =: aa_\theta\cdot C.
\end{split}
\end{equation*}
Let us simplify the bracket term $C$ above. The coefficient of $a_\theta^2$
is
\begin{equation*}
k^2\bigl((4-p)-2(p-1)\bigr) + V\bigl(2 - (4-p)\bigr) = (p-2)(V-3k^2),
\end{equation*}
and the coefficient of $k^2a^2$ is
\begin{equation*}
\begin{split}
& (4-p)(k^2-V)(p-1)-2\bigl((p-1)k^2-V\bigr) \\
&=k^2\bigl((4-p)(p-1)-2(p-1)\bigr)+V\bigl(2-(4-p)(p-1)\bigr) \\
&= k^2(p-1)(2-p)+V(p-2)(p-3).
\end{split}
\end{equation*}
We factor out $(2-p)$ to obtain $C=(2-p)D$, where
\begin{equation}
\label{J}
D := (3k^2 - V)a_\theta^2 + \left((p-1)k^2 - (p-3)V\right)k^2a^2.
\end{equation}
Hence (\ref{q_form}) reads
\begin{equation}
\label{q_theta}
q_\theta = B^{-2}A^\frac{2-p}{2}(2-p)aa_\theta D,
\end{equation}
and we deduce that the extremal points of $q$ are the points where
$a = 0$ or $a_\theta = 0$ or $D = 0$.

Differentiating (\ref{q_theta}) yields
\begin{equation*}
q_{\theta \theta} = (2-p)\bigl(aa_\theta (B^{-2}A^\frac{2-p}{2}D)_\theta + (aa_\theta)_\theta B^{-2}A^\frac{2-p}{2}D\bigr).
\end{equation*}
When $a = 0$ or $a_\theta = 0$, we have
\begin{equation*}
q_{\theta\theta} = (2-p)(a_\theta^2 - Va^2)B^{-2}A^\frac{2-p}{2}D,
\end{equation*}
and we conclude:
$q_{\theta\theta}$ has the sign of $-D$ when $a = 0$, and
$q_{\theta\theta}$ has the sign of $+D$ when $a_\theta = 0$.

Next we insert the formula for $V$ in (\ref{J}). We denote
\begin{equation*}
V = \frac{\beta a_\theta^2 + \gamma k^2a^2}{(p-1)a_\theta^2+k^2a^2},
\end{equation*}
and calculate in (\ref{J})
\begin{equation*}
\begin{split}
3k^2-V &= 3k^2 - \frac{\beta a_\theta^2+ \gamma k^2a^2}{(p-1)a_\theta^2 +
k^2a^2} \\
&= \frac{3k^2\bigl((p-1)a_\theta^2+k^2a^2\bigr)-\beta a_\theta^2-\gamma k^2a^2}{(p-1)a_\theta^2 + k^2a^2} \\
&= \frac{a_\theta^2\bigl(3k^2(p-1)-\beta\bigr)+k^2a^2(3k^2-\gamma)}{(p-1)a_\theta^2 + k^2a^2},
\end{split}
\end{equation*}
and 
\begin{equation*}
\begin{split}
& (p-1)k^2-(p-3)V = (p-1)k^2 - (p-3)\frac{\beta a_\theta^2+ \gamma k^2a^2}{(p-1)a_\theta^2 + k^2a^2} \\
&= \frac{(p-1)k^2\bigl((p-1)a_\theta^2+k^2a^2\bigr)-(p-3)\bigl(\beta a_\theta^2+\gamma k^2a^2\bigr)}{(p-1)a_\theta^2 + k^2a^2} \\
&= \frac{a_\theta^2\bigl((p-1)^2k^2-(p-3)\beta\bigr)+k^2a^2\bigl((p-1)k^2-(p-3)\gamma\bigr)}{(p-1)a_\theta^2 + k^2a^2}.
\end{split}
\end{equation*}
Further, inserting $\beta = (2p-3)k^2 - (p-2)k$ and $\gamma = (p-1)k^2-(p-2)k$
in the nominator yields
\begin{equation*}
\begin{split}
& 3k^2(p-1)-\beta = p(k^2+k)-2k, \\
& 3k^2-\gamma = p(-k^2+k) + 4k^2-2k, \\
& (p-1)^2k^2 - (p-3)\beta = p^2(-k^2+k)+p(7k^2 - 5k)+(-8k^2+6k), \\
& (p-1)k^2-(p-3)\gamma = p^2(-k^2+k)+p(5k^2-5k)+(-4k^2+6k).
\end{split}
\end{equation*}
Disregarding the positive denominator, we have that $D$ in (\ref{J}) has the
same sign as the expression
\begin{equation}
\label{magic}
\begin{split}
& a_\theta^4\bigl(p(k^2+k)-2k\bigr) \\
& + a_\theta^2 k^2a^2\bigl(p^2(-k^2+k)+p(6k^2-4k)+(-4k^2+4k)\bigr) \\
& + k^4a^4\bigl(p^2(-k^2+k) + p(5k^2-5k) + (-4k^2+6k)\bigr),
\end{split}
\end{equation}
which factorizes to
\begin{equation*}
\bigl([p(k^2+k)-2k]a_\theta^2 + [p^2(-k^2+k)+p(5k^2-5k)+(-4k^2+6k)]k^2a^2\bigr)\bigl(a_\theta^2+k^2a^2).
\end{equation*}
Modifying further, this becomes
\begin{equation}
\label{factorizes}
\begin{split}
& \bigl((k+1)p - 2\bigr)a_\theta^2 +
\bigl((-k + 1)p^2 + (5k - 5)p -4k + 6\bigr)k^2a^2 \\
=&  \bigl((k+1)p - 2\bigr)a_\theta^2 +
\bigl((-k + 1)(p-p_1)(p-p_2)\bigr)k^2a^2,
\end{split}
\end{equation}
where
\begin{equation*}
p_{1,2} = \frac{5}{2} \pm \frac{1}{2}\sqrt\frac{9k-1}{k-1}
\end{equation*}
with the convention $p_1 < p_2$. Factoring out $k-1$ in (\ref{factorizes})
finally yields that $D$ has the same sign as
\begin{equation*}
\left(\frac{k+1}{k-1}p -\frac{2}{k-1}\right)a_\theta^2 - (p-p_1)(p-p_2)k^2a^2.
\end{equation*}
We note that the coefficient of $a_\theta^2$ is positive and that
$p_1 < 1$. Thus the sign of $D$ is positive whenever $p < p_2$.
In this case we observe that the local minimum of $q$ occurs when $a_\theta = 0$, and
we have $(\tau q)_\theta > 0$ as desired.

The remaining case to check is $p \geq p_2$ and $D \leq 0$, where $D\leq 0$ reads
\begin{equation*}
\bigl((k+1)p-2\bigr)a_\theta^2 \leq \bigl((k-1)p^2 + (-5k + 5)p + (4k-6)\bigr)k^2a^2,
\end{equation*}
i.e.
\begin{equation*}
a_\theta^2 \leq \frac{(k-1)p^2 + (-5k+5)p + (4k-6)}{(k+1)p -2}k^2a^2.
\end{equation*}
Denote the fraction on the right-hand side by $F$. It suffices to show (for
$k \geq 2$) that $F < p-1$ when $p\geq p_2$, since
$a_\theta^2 < (p-1)k^2a^2$ yielded
$(\tau q)_\theta > 0$ as desired. Now $F < p-1$ precisely when
\begin{equation*}
p^2 + (2k-4)p -2k+4 > 0,
\end{equation*}
in particular whenever
\begin{equation*}
p > \sqrt{k^2 - 2k} -k + 2.
\end{equation*}
But for $k \geq 2$ we have $\sqrt{k^2 - 2k} -k+2 < 4$, while $p \geq p_2 > 4$.
This completes the proof.
\end{proof}

\subsection{Calculations for the Dirichlet form}
\label{on_dirichlet}
In this subsection we provide the calculations missing from the proof of Lemma \ref{oblique_lemma} above.

\begin{lemma}
\label{calculate2}
Let $A$ be as in (\ref{A}) and write $A = \begin{pmatrix}a_{11}& a_{12}\\ a_{21} & a_{22}\end{pmatrix}$. Then 
\begin{equation*}
\begin{split}
& \operatorname{div}^\circ (A\nabla^\circ u) \\
& = a_{11}e_re_r(u) + \bigl(a_{12} + a_{21}\bigr)e_\theta e_r(u) + a_{22}e_\theta e_\theta(u) \\
& \quad + \left(e_r(a_{11}) + \frac{1}{r}a_{11} +  e_\theta(a_{21})\right)e_r(u)
 + \Bigl(e_r(a_{12}) +  e_\theta(a_{22})\Bigr)e_\theta (u).
\end{split}
\end{equation*}
\end{lemma}
\begin{proof}
We calculate
\begin{equation*}
\label{calculate}
\begin{split}
& \operatorname{div}^\circ (A\nabla^\circ u)
= \frac{1}{r}e_r\bigl(r\left[A\nabla^\circ u\right]_1\bigr) + e_\theta\bigl(\left[A\nabla^\circ u\right]_2\bigr) \\
& = a_{11}e_re_r(u) + e_r(a_{11})e_r(u) + a_{12}e_re_\theta(u)
+ e_r(a_{12})e_\theta(u) \\
&\quad + \frac{1}{r}a_{11}e_r(u) + \frac{1}{r}a_{12}e_\theta(u) \\
&\quad + a_{21}e_\theta e_r(u)+e_\theta (a_{21})e_r(u) + a_{22}e_\theta e_\theta (u) + e_\theta (a_{22})e_\theta (u)
\\
& = a_{11}e_re_r(u) + a_{12}e_re_\theta (u) + a_{21}e_\theta e_r(u) + a_{22}e_\theta e_\theta(u) \\
& \quad + \left(e_r(a_{11}) + \frac{1}{r}a_{11} +  e_\theta(a_{21})\right)e_r(u)
+ \left(e_r(a_{12}) + \frac{1}{r}a_{12} +  e_\theta(a_{22})\right)e_\theta (u).
\end{split}
\end{equation*}
Moreover,
\begin{equation*}
\begin{split}
e_re_\theta(u) & = e_r\bigl(\frac{1}{r}\frac{\partial u}{\partial\theta}\bigr)
= -\frac{1}{r^2}\frac{\partial u}{\partial \theta} + \frac{1}{r}e_r\bigl(\frac{\partial u}{\partial \theta}\bigr) \\
& = -\frac{1}{r^2}e_\theta(u) + \frac{1}{r}\frac{\partial}{\partial r}\bigl(\frac{\partial u}{\partial \theta}\bigr)
 = -\frac{1}{r^2}e_\theta(u) + \frac{1}{r}\frac{\partial}{\partial \theta}\bigl(\frac{\partial u}{\partial r}\bigr) \\
& = -\frac{1}{r^2}e_\theta(u) + e_\theta e_r(u),
\end{split}
\end{equation*}
so
\begin{equation*}
a_{12}e_r e_\theta(u) = a_{12}e_\theta e_r(u) - \frac{1}{r}a_{12}e_\theta(u),
\end{equation*}
and the claim follows.
\end{proof}

\begin{lemma}
\label{stokes_thm}
Let $A$ be as in (\ref{A}), let $\mathcal{C}$ be as in (\ref{C}), and
let $B = A + \mathcal{C}$. Denote the conormal
derivative with respect to $B$ on $\partial \D$ of a function $u$ by
\begin{equation*}
\frac{\partial u}{\partial n_B^*} = \langle B\begin{pmatrix} 1 \\ 0 \end{pmatrix}
, \nabla^\circ u \rangle.
\end{equation*}
Then
\begin{equation}
\label{stokes}
\int_{\partial\D}v\frac{\partial u}{\partial n_B^*}d\theta
= \int_\D v\operatorname{div}^\circ (B\nabla^\circ u)\,dA
+ \int_\D\langle\nabla^\circ v, B\nabla^\circ u\rangle\,dA \\
\end{equation}
for each $u,v\in Y_1$.
\end{lemma}
\begin{proof}
By Lemma \ref{dense}, we may assume that $u,v \in C^1(\D)$. By definition,
\begin{equation}
\label{74}
\int_\D v\operatorname{div}^\circ U\,dA = -\int_\D \langle U, \nabla^\circ v\rangle\,dA
\end{equation}
for each $U \in C^1(\D; \R^2)$ and $v\in C^\infty_0(\D)$.
When $v$ is not compactly supported, we multiply it by $\varphi_\varepsilon$,
a standard radial function in $C^\infty_0(\D)$ satisfying
$\varphi_\varepsilon \to \chi_\D$ as $\varepsilon \to 0$. Then, by (\ref{74}),
\begin{equation*}
\begin{split}
\int_\D \varphi_\varepsilon v\operatorname{div}^\circ U\,dA & = -\int_\D \langle U, \nabla^\circ (\varphi_\varepsilon v)\rangle\,dA \\
& = - \int_\D \langle U , v\nabla^\circ \varphi_\varepsilon\rangle\,dA - \int_D\langle U ,
\varphi_\varepsilon\nabla^\circ v\rangle\,dA.
\end{split}
\end{equation*}
Letting $\varepsilon \to 0$ yields, since $\nabla^\circ \varphi_\varepsilon \to (-\delta_1, 0)$ as $\varepsilon \to 0$,
\begin{equation*}
\int_\D v\operatorname{div}^\circ U\,dA  = \int_{\partial\D} U_1 v\,d\theta - \int_\D \langle U , \nabla^\circ v\rangle\,dA.
\end{equation*}
With $U = B\nabla^\circ u$, we have $U_1 = b_{11}e_r(u) + b_{12}e_\theta(u)$,
and
\begin{equation*}
\frac{\partial u}{\partial n^*_B} = \langle B\begin{pmatrix}1\\0\end{pmatrix},\nabla^\circ u\rangle  = U_1,
\end{equation*}
which finishes the proof.
\end{proof}

\begin{lemma}
\label{the_form}
The Dirichlet form \eqref{form_definition} satisfies \eqref{dirichletform}.
\end{lemma}
\begin{proof}
Replacing the divergence term in (\ref{stokes}) by Lemma \ref{calculate2} yields
\begin{equation}
\label{iiiii}
\begin{split}
& \int_{\partial\D}v\frac{\partial u}{\partial n_B^*}d\theta
= \int_\D\nabla^\circ v\cdot B\nabla^\circ u\,dA \\
& +\int_\D v\left\{ b_{11}e_re_r(u) + \bigl(b_{12} + b_{21}\bigr)e_\theta e_r(u) + b_{22}e_\theta e_\theta(u)\right\}\,dA \\
& + \int_\D v\left\{\left(e_r(b_{11}) + \frac{1}{r}b_{11} +  e_\theta(b_{21})\right)e_r(u) \right. \\
 & \qquad \left. + \Bigl(e_r(b_{12}) +  e_\theta(b_{22})\Bigr)e_\theta (u)\right\}\,dA,
\end{split}
\end{equation}
and replacing the middle term on the right-hand side of (\ref{iiiii})
by 
\begin{equation*}
%\label{TU_rewrite}
\begin{split}
& \operatorname{div}^\circ (A\nabla^\circ u) \\
& = b_{11}e_re_r(u) + \bigl(b_{12} + b_{21}\bigr)e_\theta e_r(u) + b_{22}e_\theta e_\theta(u) \\
& \quad + \left(e_r(a_{11}) + \frac{1}{r}a_{11} +  e_\theta(a_{21})\right)e_r(u)
 + \Bigl(e_r(a_{12}) +  e_\theta(a_{22})\Bigr)e_\theta (u).
\end{split}
\end{equation*}
yields
\begin{equation*}
\begin{split}
& \int_{\partial\D}v\frac{\partial u}{\partial n_B^*}d\theta
 = \int_\D\nabla^\circ v\cdot B\nabla^\circ u\,dA + \int_\D v \operatorname{div}^\circ (A\nabla^\circ u)\,dA \\
& + \int_\D v\left\{\left(e_r(b_{11}) + \frac{1}{r}b_{11} +  e_\theta(b_{21})\right)e_r(u) + \Bigl(e_r(b_{12}) +  e_\theta(b_{22})\Bigr)e_\theta (u) \right\}\,dA\\
& - \int_\D v\left\{ \left(e_r(a_{11}) + \frac{1}{r}a_{11} +  e_\theta(a_{21})\right)e_r(u) + \Bigl(e_r(b_{12}) +  e_\theta(b_{22})\Bigr)e_\theta (u) \right\}\,dA. 
\end{split}
\end{equation*}
Thus we obtain (since $c_{ij} = b_{ij}-a_{ij}$, $c_{11} = c_{22} = 0$, and $c_{21} = -c_{12} = c$)
\begin{equation*}
\begin{split}
& \int_{\partial\D}v\frac{\partial u}{\partial n_B^*}d\theta
= \int_\D\nabla^\circ v\cdot B\nabla^\circ u\,dA + \int_\D v \operatorname{div}^\circ (A\nabla^\circ u)\,dA \\
& + \int_\D v\bigl\{e_\theta(-c)e_r(u)
 + e_r(c)e_\theta (u) \bigr\}\,dA.
\end{split}
\end{equation*}
We want
\begin{equation*}
\frac{\partial u}{\partial n_B^*} = \frac{\partial u}{\partial n_A^*} + \tau(\theta)\frac{\partial u}{\partial \theta}.
\end{equation*}
Since
\begin{equation*}
\frac{\partial u}{\partial n_B^*} = (A + \mathcal{C})^t\begin{pmatrix}1\\0\end{pmatrix} \cdot \nabla^\circ u
= \frac{\partial u}{\partial n_A^*} + \mathcal{C}^t\begin{pmatrix}1\\0\end{pmatrix} \cdot \nabla^\circ u,
\end{equation*}
and since
\begin{equation*}
\mathcal{C}^t\begin{pmatrix}1\\0\end{pmatrix} \cdot \nabla^\circ u = 
\begin{pmatrix}
0 & c \\
-c & 0
\end{pmatrix}
\begin{pmatrix}
1\\
0
\end{pmatrix}
\cdot  
\begin{pmatrix}
e_r(u) \\
e_\theta(u)
\end{pmatrix}
= -c\,e_\theta(u) = -c\frac{\partial u}{\partial \theta},
\end{equation*}
we choose $c$ to be any function in $C^\infty(\overline{\D})$ such that
$-c(1,\theta) = \tau(\theta)$. The additional condition $c(r,\theta) = 0$ for $r < 1/2$ is needed both for Lemma \ref{stokes_thm} above and for the
coercivity estimate below.
\end{proof}

\begin{lemma} 
There exist constants $C_1$ and $C_2$, independent of $u$, such that
\begin{equation*}
|D(u,u)| \geq C_1||u||^2_{Y_1} - C_2||u||^2_{Y_0}
\end{equation*}
for each $u \in Y_1$.
\end{lemma}
\begin{proof}
We estimate
\begin{equation}
\label{coer_est}
|D(u,u)| \geq \left| \int_\D\nabla^\circ u\cdot B\nabla^\circ u\,dA \right| 
 -\left| \int_\D u\bigl\{e_\theta(c_{21})e_r(u)
 + e_r(c_{12})e_\theta (u) \bigr\}\,dA \right|.
\end{equation}
Since $\nabla^\circ u\cdot B\nabla^\circ u = \nabla^\circ u\cdot A \nabla^\circ u$
and since $A\xi\cdot\xi \geq Cr^{2\alpha}|\xi|^2$,
we have
\begin{equation*}
\int_\D\nabla^\circ u\cdot B\nabla^\circ u\,dA  \geq 
C\int_\D r^{2\alpha}|\nabla^\circ u|^2\,dA = C\bigl(||u||^2_{Y_1} - ||u||^2_{Y_0}\bigr).
\end{equation*}
For the second term on the right-hand side of (\ref{coer_est}) we
have
\begin{equation*}
\int_\D u\bigl\{e_\theta(c_{21})e_r(u) + e_r(c_{12})e_\theta (u) \bigr\}\,dA
\leq C\int_{\{r\geq \frac{1}{2}\}} |u\bigl(e_r(u) + e_\theta(u)\bigr)|\,dA,
\end{equation*}
since the functions $c_{ij}$ are in the class $C^\infty(\overline{\D})$ and
are supported in the annulus $\{r\geq 1/2\}$.
By Young's inequality,
\begin{equation*}
\begin{split}
& \int_{\{r\geq \frac{1}{2}\}} |u\bigl(e_r(u) + e_\theta(u)\bigr)|\,dA 
\leq C\int_{\{r\geq \frac{1}{2}\}} |u\nabla^\circ u|\,dA \\
& \leq C\varepsilon \int_{\{r\geq \frac{1}{2}\}} |\nabla^\circ u|^2\,dA
+ C\frac{1}{\varepsilon} \int_{\{r\geq \frac{1}{2}\}} |u|^2\,dA 
 \leq C\varepsilon ||u||_{Y_1}^2 + C\frac{1}{\varepsilon}||u||_{Y_0}^2.
\end{split}
\end{equation*}
Choose $\varepsilon>0$ small enough such that $C\varepsilon \leq 1/2$
to obtain
\begin{equation*}
\begin{split}
& |D(u,u)| \geq \left| \int_\D\nabla^\circ u\cdot B\nabla^\circ u\,dA \right| 
 -\left| \int_\D u\bigl\{e_\theta(c_{21})e_r(u)
 + e_r(c_{12})e_\theta (u) \bigr\}\,dA \right| \\
& \geq C||u||_{Y_1}^2 - \frac{1}{2}||u||_{Y_1}^2 - C_2||u||_{Y_0}^2
=  C_1||u||_{Y_1}^2 - C_2||u||_{Y_0}^2,
\end{split}
\end{equation*}
as wanted.
\end{proof}

Finally, we prove a result about the adjoint Dirichlet form $D^*(v,u) = D(u,v)$ that is needed in the proof of Lemma \ref{oblique_lemma}.
\begin{lemma}
If $u \in Y_1$ satisfies, for some $f \in Y_0^*$,
\begin{equation*}
D^*(v,u) = \langle v\,|\, f\rangle  \qquad\text{ for all } v\in Y_1,
\end{equation*}
then u is a weak solution to the boundary value problem
\begin{equation}
\label{adjoint_}
\begin{cases}
Tu &= f \qquad \text{in } \D \\
\frac{\partial u}{\partial n^*} - \frac{\partial}{\partial \theta}(\tau u ) &= 0 \qquad \text{on } \partial\D.
\end{cases}
\end{equation}
\end{lemma}
\begin{proof}
Since $A$ is symmetric, Lemma \ref{stokes_thm} yields
\begin{equation}
\label{Green}
\int_{\D}vTu - uTv \,dA = \int_0^{2\pi} u(1,\theta)\frac{\partial v}{\partial n^*}(1,\theta)
- v(1,\theta)\frac{\partial u}{\partial n^*}(1,\theta)\, d\theta.
\end{equation}
By definition of $D(v,u)$,
\begin{equation*}
D(v,u) = \int_{\D}vTu\,dA + \int_0^{2\pi} v(1,\theta)\left( \frac{\partial u}{\partial n^*}(1,\theta) + \tau(\theta)\frac{\partial u}{\partial \theta}(1,\theta)\right)\,d\theta,
\end{equation*}
and
\begin{equation*}
D^*(v,u) = \int_{\D}uTv\,dA + \int_0^{2\pi} u(1,\theta)\left( \frac{\partial v}{\partial n^*}(1,\theta) + \tau(\theta)\frac{\partial v}{\partial \theta}(1,\theta)\right)\,d\theta,
\end{equation*}
so combined with (\ref{Green}),
\begin{equation*}
D^*(v,u) - D(v,u) = \int_0^{2\pi} u(1,\theta)\tau(\theta)\frac{\partial v}{\partial n^*}
(1,\theta) - v(1,\theta)\tau(\theta)\frac{\partial u}{\partial n^*}(1,\theta)\,d\theta.
\end{equation*}
Thus $D^*(v,u)$ and $D(v,u)$ differ only on the boundary, and the boundary
condition for $D^*$ is
\begin{equation*}
\int_0^{2\pi} v\frac{\partial u}{\partial n^*} + u\tau\frac{\partial v}{\partial \theta}\,d\theta
= \int_0^{2\pi} v\left(\frac{\partial}{\partial n^*} - \frac{\partial}{\partial \theta}(\tau u)\right)\,d\theta.
\end{equation*}
\end{proof}

\section{Related problems}
%\label{section8}

We close with some problems listed by Wolff in
\cite{Wolff:1984}, where progress has since been made:

\label{Lewis}
1. {\it Are there bounded $p$-harmonic functions with bad behavior at every
point on the boundary and if not, is a Fatou theorem true if one interprets
``almost everywhere'' using a finer measure?} These questions were
answered by \mbox{Manfredi} and Weitsman \cite{Manfredi-Weitsman:1988} in
1988, see also \cite{FGM-MS:1988}. The Hausdorff dimension
of the set (on the boundary of a smooth Euclidean domain)
where radial limits exist is bounded below with a positive constant that
depends only on the number $p$ and the dimension of the underlying space.
No estimates for this constant are known even in the plane.

2. {\it What can be said about radial limits of quasiregular
mappings?} Wolff states \cite[p. 373]{Wolff:1984} that this question was
the main motivation for his work. Progress was made by
K.~Rajala \cite{K.Rajala:2008}: If a quasiregular mapping $\B^n \to \R^n$
is a local homeomorphism, then radial limits exist at infinitely many
boundary points. Apart from this result, the question seems to be open.

\section*{Acknowledgements}
This problem was proposed to me by Professor Juan J. Manfredi; I want to thank him for
being an excellent mentor and host. I also thank 
Tuomo Kuusi, Eero Saksman and Xiao Zhong for interest and encouragement. Finally, I thank two anonymous referees
for the careful reading of previous manuscripts.

%%%%%%%%%%%%%%%%
% bibliography
%%%%%%%%%%%%%%%

% Set bibliography items using the "thebibliography" environment  and following
% the style used by the AMS journals. 
%
% If the bibliography is generated by a bibtex database, use "amsplain" or
% "amsalpha" as bibliography style

% Bibliography generated (amsplain style) and modified by hand
\providecommand{\bysame}{\leavevmode\hbox to3em{\hrulefill}\thinspace}
\providecommand{\MR}{\relax\ifhmode\unskip\space\fi MR }
% \MRhref is called by the amsart/book/proc definition of \MR.
\providecommand{\MRhref}[2]{%
  \href{http://www.ams.org/mathscinet-getitem?mr=#1}{#2}
}
\providecommand{\href}[2]{#2}

\end{document}